\documentclass[12pt,a4paper,final]{amsart}
\usepackage[utf8]{inputenc}
\usepackage{amsthm}          %
\usepackage{amsmath}
\usepackage{amssymb}
\usepackage{buczofpc}
\usepackage{titletoc}
\DeclareRobustCommand{\SkipTocEntry}[5]{}
\usepackage[bookmarks=true,colorlinks,linkcolor=magenta,urlcolor=magenta,citecolor=magenta]{hyperref}
\usepackage[capitalize,noabbrev]{cleveref}
\crefformat{equation}{(#2#1#3)}
\crefformat{figure}{#2Figure #1#3}
\crefformat{enumi}{(#2#1#3)}
\crefformat{theorem}{#2Theorem~#1#3}
\crefformat{lemma}{#2Lemma~#1#3}
\crefformat{proposition}{#2Proposition~#1#3}
\crefformat{corollary}{#2Corollary~#1#3}
\crefformat{claim}{#2Claim~#1#3}
\crefformat{section}{#2§#1#3}
\crefformat{subsection}{#2§#1#3}
\crefalias{theorem}{black} 
\crefname{theorem}{\textcolor{black}{Theorem}}{\textcolor{black}{Theorems}}

\usepackage{enumitem}
\usepackage[dvips]{graphicx}
\usepackage{transparent}
\usepackage{wrapfig}
\usepackage[usenames,dvipsnames,svgnames,table]{xcolor}

\usepackage{color}
\usepackage{type1cm}         %
\usepackage[english]{babel}  %
\usepackage[bf]{caption}     
\usepackage[top=1.5cm, bottom=1.5cm]{geometry}        
\usepackage{bm}  
\usepackage{version} 
\usepackage[notcite,notref]{showkeys} 
\usepackage{showkeys} 
\usepackage{datetime}

\RequirePackage[hyperpageref]{backref}

\renewcommand*{\backref}[1]{}
\renewcommand*{\backrefalt}[4]{
\ifcase #1 (not cit.)
\or (cit. on p. #2)
\else (cit. on pp. #2)\fi}

\definecolor{aquam}{rgb}{0.5,1.0,1.0}
\definecolor{bbrown}{rgb}{0.75,0.38,0.15}
\definecolor{Cyan}{rgb}{0,0.6,0.6}
\definecolor{Darkblue}{rgb}{0,0,1}
\definecolor{Dodgerblue2}{rgb}{0,0.5,1}
\definecolor{Green}{rgb}{0,0.3,0.06}
\definecolor{Kahki}{rgb}{1,1,0.5}
\definecolor{Magenta}{rgb}{1,0,1}
\definecolor{bMagenta}{rgb}{1,.6,1}
\definecolor{Orange}{rgb}{0.8,0.3,0}
\definecolor{dOrchid}{rgb}{0.7,0.2,0.4}
\definecolor{Orchid}{rgb}{1,0.5,1}
\definecolor{Purple}{rgb}{0.65,0.07,0.85}
\definecolor{Royalblue}{rgb}{0.6,0.85,0.87}
\definecolor{Tan}{rgb}{0.54,0.42,0.23}
\definecolor{bTan}{rgb}{0.94,0.82,0.63}
\definecolor{zoltan}{rgb}{0,0.1,0.3}
\definecolor{Turquoise}{rgb}{0,0.85,0.87}
\definecolor{Yellow}{rgb}{1,1,0}
\definecolor{bYellow}{rgb}{1,1,0.6}
\definecolor{bRed}{rgb}{1,0.7,0.7}
\definecolor{boxcolb}{rgb}{0.87,0.77,0.75}
\definecolor{boxcol}{rgb}{0.6,0.85,0.87}
\definecolor{boxcolgreen}{rgb}{0.64,0.93,0.79}
\definecolor{boxcolaa}{rgb}{.75,.99,.70}
\definecolor{boxcolbb}{rgb}{0.39,0.50,0.56}
\definecolor{boxcolcc}{rgb}{1,0.81,0.65}
\definecolor{yy}{rgb}{0.43,0.21,.18}
\definecolor{gA}{gray}{0.5}
\definecolor{gB}{gray}{0.8}
\definecolor{gC}{gray}{0.9}

\numberwithin{equation}{section}

\theoremstyle{plain}
\newtheorem{theorem}{Theorem}[section]
\newtheorem{corollary}[theorem]{Corollary}
\newtheorem{proposition}[theorem]{Proposition}
\newtheorem{lemma}[theorem]{Lemma}

\theoremstyle{remark}
\newtheorem{remark}[theorem]{Remark}

\theoremstyle{definition}
\newtheorem{definition}[theorem]{Definition}

\def\mg{{\mathbf g}}

\def\H{{\mathcal H}}

\def\P{{\mathbb P}_p}
\def\Q{{\mathcal Q}}
\def\C{{\mathcal C}}

\def\K{{\mathcal K}}
\def\R{{\mathbb R}}

\def\N{{\mathbb N}}
\def\T{{\mathcal T}}

\def\ii{{\bf i}}

\def\bL{{\bf L}}

\newcommand{\mfc}{\mathfrak c}

\DeclareMathOperator{\dimh}{dim_H}

\DeclareMathOperator{\dist}{dist}

\DeclareMathOperator{\proj}{proj}

\begin{document}

\title[Pure unrectifiability, fractal percolation and graphs of functions]
{Purely unrectifiable sets, fractal percolation and graphs of functions}

\author[Z. Buczolich]{Zolt\'an Buczolich$^1$}
\address{Department of Analysis, ELTE E\"otv\"os Lor\'and University, 
         P\'azm\'any P\'eter s\'e\-t\'any 1/c, H-1117 Budapest, Hungary}
\email{zoltan.buczolich@ttk.elte.hu}


\thanks{The author thanks Westlake University, Hangzhou, China for the invitation to the Winter School and to prepare this survey.}

\begin{abstract} 
 
This paper contains a survey of some of the results of the author related to unrectifiablity and is an extended version of the author's talk given at 
the Second Winter School
Geometric Measure Theory
Rectifiability vs. Pure Unrectifiability in 
Hanghzou, China.
 These results include irregular/purely unrectifiable  $1$-sets on the graphs of continuous functions like the Takagi, the Weierstrass-Cellerier and the typical (in the sense of Baire) continuous function. It is also discussed that there exists $\aaa_{0}<1$ such that the fractal percolation is 
almost surely purely $\alpha$-unrectifiable for all $\alpha>\alpha_0$. The background of the $1$-unrectifiability is discussed in more detail.


\end{abstract}

\maketitle

\tableofcontents

\section{Introduction}\label{intro}

This paper contains a survey of some of the results of the author related to unrectifiablity and is an extended version of the author's talk given at the
Second Winter School
Geometric Measure Theory
Rectifiability vs. Pure Unrectifiability held at
Westlake University, Hanghzou, China, February 1 - 6, 2026.

The first part of this survey paper is based on \cite{Buirr} which is joint with 
 Esa and Maarit J\"arvenp\"a\"a, Tam\'as Keleti, and Tuomas
Pöyht\"ari.
 After some definitions in Section \ref{secunrect}, in Section \ref{model} we deal with the result about  almost sure pure $\aaa$-unrectifiability of fractal percolation. This requires some definitions and preliminary results in Subsection \ref{deffracperc}, which are followed by the statement of Theorem 
 \ref{maingeneral} which is the  main result of Section \ref{model} and states that there exists $\aaa_{0}<1$ such that the fractal percolation is 
almost surely purely $\alpha$-unrectifiable for all $\alpha>\alpha_0$.
The proof of this Theorem in \cite{Buirr} contains many diffucult and tricky technical details. Hence  in Subsection \ref{subsecpureunr}   we discuss the ideas behind the special case when one considers $1$-unrectifiability, which means unrectifiability in the usual sense and to make things even simpler we deal with only the two-dimensional case.  Section \ref{model} is closed with Subsection \ref{secghold} which contains a few remarks about the background of the proof of Theorem \ref{maingeneral}.

The second part of this paper is Section \ref{secregoncont}    which is devoted to regular and irregular sets on the graphs of continuous functions. 
By the Bolzano--Darboux property there can be at most one direction in which the projection of a continuous function defined on an interval does not contain an interval. Hence the ``regular'' part of the graph of a continuous function is always non-empty. It is much more interesting to find irregular sets on the graphs of fractal continuous functions.
In Subsection \ref{subsectak} the main result is Theorem \ref{thtakirrdec} which gives a natural decomposition of the graph of the Takagi function into regular and irregular parts. 

In the next subsection we consider typical/generic continuous functions in the Baire category sense. The key definitions are 
Definition \ref{subsecmicrocont} 
of the micro tangent sets and Definition \ref{UMTdef} of $UMT(f)$ the set of the universal-micro tangent points. In this subsection the main results are Theorems \ref{umt} and \ref{piy} telling that for almost every $x$ the point $(x;f(x))$ is in $UMT(f)$, but the projection of $UMT(f)$ onto the $y$-axis is of Lebesgue measure zero, in fact, the projection of this set is of zero measure into any other direction different from the projection onto the $x$-axis. This way we obtain a naturally defined irregular $1$-set on the graph of a typical continuous function.   This subsection is mainly based on the paper \cite{Bumicro}.

Theorem \ref{th2} shows that ``most'' points on the graph of the typical continuous function  in the sense of one-dimensional Hausdorff measure  are not in $UMT(f)$. Hence it is interesting to consider on the graph of the typical continuous function the other, non-$UMT$ points. This is the topic of Subsection \ref{subsecmicrocontvert}, which is based on the paper \cite{BuRa}. The concept of vertical universal points, $VMT(f)$ is introduced. According to Theorems \ref{T2}, \ref{thvmt} and \ref{T4}  at most points we have vertical universality, but as Theorem \ref{T3} shows there are still exceptional points.

Subsection \ref{subsecmicrotak} contains some remarks about the micro tangent sets of the Takagi function and of the $1D$ Brownian motion. In Theorem \ref{takagi} we see that at almost every $x$ there is a nice micro self-similarity property of the Takagi function while Theorem \ref{bmo} tells that the Brownian motion is too wild, with probability one its universal micro tangent set is empty and the central component of its micro tangent set is just one vertical line segment.

Subsection \ref{subsecweierstrass}  is devoted to the Weierstrass--Cellerier  nowhere differentiable function and is based on papers \cite{Buirra} and \cite{buczoccupation2010}. In Theorem \ref{irregthm} we mention that a decomposition into regular and irregular part of its graph can be obtained. This is analogous to the one obtained in Subsection \ref{subsectak} for the Takagi function. Cardinality of level sets and occupation measures also play an important role in different sections of this survey. In earlier sections it is mentioned that almost every level set of the Takagi function is finite and its occupation measure is singular with respect to the Lebesgue measure. Level sets of the typical continuous function are much larger, but its occupation measure is still singular. On the other hand, the Brownian motion with probability one has absolutely continuous occupation measure. In Subsection \ref{subsecweierstrass} a class of perturbed Weierstrass--Cellerier functions is introduced and in Theorem \ref{flevthm} we see that if a function belongs to this class then singularity of the occupation measure implies finiteness of almost every level set. In Theorem \ref{thmmainwoc} we recall that for the class ${{\mathcal W}}(x,c) {{=}} {{\mathcal W}}(x)+cx$ the occupation measure is singular for all $c$s.

The final subsection is based on \cite{buczolich2025levelsets}, a very recent joint paper with  Antti K\"aenm\"aki, and Bal\'azs Maga. Level sets and occupation measures of prevalent Weierstrass functions are considered. 
Theorem \ref{*prabscb} shows that a parametrized  class $W_ {\mathbf {t}}$ has absolutely continuous occupation measure with density in $L^2( {\mathbb {R}})$ for almost every parameter $\mathbf {t}$.
The main theorems in this section are Theorems \ref{thm:main1}  \ref{thm:main2} about dimensions of level sets of prevalent Weierstrass functions.
Key tool is Theorem \ref{prop:bi-holder} about the existence of Weierstrass embeddings.

\section{Unrectifiability}\label{secunrect}

The natural numbers are denoted by $\N=\{1,2,...\}$.

Given $A {\subset}  {\ensuremath {\mathbb R}}^{d}$, by $|A|$,  $int(A)$, and $cl(A)$
we mean its diameter, interior, and closure, respectively. The Lebesgue measure of $A\sse \R^d$ is denoted by ${\mathcal {L}}^d(A)$, in the one-dimensional case we use ${\mathcal {L}}(A)$ instead of ${\mathcal {L}}^1(A)$. 

The $k$-dimensional Hausdorff measure is denoted by $\cah^{k}$, the Hausdorff dimension of $A$ is denoted by $\dimh A$. For the definition of the Hausdorff measure and for some other basic concepts of fractal geometry we refer to \cite{FalFG}.
First we recall the definitions of rectifiability and unrectifiability.

 \begin{definition} \label{defrectif}  
  A Borel set $E\subset \mathbb R^d$ is {\it a rectifiable subset of dimension $k$} if $\dimh E=k$  and there exists   a countable family of Lipschitz maps $f_i: \mathbb R^k \to \mathbb R^d$ such that  
$\ds {\cah}^{k}(E\sm \cup f_{i}(\R^{k}))=0$, that is their images cover $\mathcal{H}^k$-almost all of $E$. 
 \end{definition} 
 
The classical planar case, with one dimensional Lipschitz curves was studied by {Besicovitch}  see \cite{Besi28,Besi38,Besi} and see also Falconer's book \cite{Fal1986}. For the higher dimensional case we refer to {Mattila}'s monograph \cite{M}.

 \begin{definition}  \label{defpurunrect}
 A Borel set of Hausdorff dimension $k$ which is not rectifiable is called {\it unrectifiable}.  
 
An unrectifiable $k$-dimensional set $E\subset \mathbb R^d$ is called  {\it purely $k$-unrectifiable} if its intersection with any $k$-dimensional rectifiable set is an $\mathcal{H}^k$-null set.
 \end{definition}
 
Previous definitions imply that  an unrectifiable set is purely unrectifiable if and only if its intersection with the image of an arbitrary Lipschitz map $f:\mathbb R^k\to \mathbb R^d$  is an $\mathcal{H}^k$-null set.

The following characterisation of pure $k$-unrectifiability can be found in \cite[Theorem 15.21]{M} or \cite[Theorem 3.2.29]{Fed}.

\begin{theorem}\label{FM}
Let $k\in\N$.  A set $F\subset\R^d$ is purely 
$k$-unrectifiable if $\H^k(M\cap F)=0$ for all $k$-dimensional 
$C^1$-submanifolds $M\subset\R^d$.
\end{theorem}

The following  is a simple consequence of the above and of  Fubini's theorem.

\begin{corollary}\label{1impliesk}
If $F\subset\R^d$ is a purely $1$-unrectifiable Borel set,  then it is
also purely $k$-unrectifiable for any $k\in \N$.
\end{corollary}

In \cite{Buirr} we considered 
$\alpha$-H\"older curves instead of
 Lipschitz or $C^{1}$ ones.

Next we suppose that $I\subset\R$ is a closed and bounded
interval. 

\begin{definition} \label{def:holdal} 
Assume that $0<\alpha\le 1$ and $H\ge 0$. 
A curve $\gamma\colon I\to\R^d$ is {\it $(H,\alpha)$-H\"older  at 
$a\in I$}, if  for every $b\in I$,
\begin{equation*}
|\gamma(a)-\gamma(b)|\le H|a-b|^\alpha.
\end{equation*}
 A curve $\gamma$ is {\it $(H,\alpha)$-H\"older}, if it 
is $(H,\alpha)$-H\"older at every $a\in I$. A curve $\gamma$ is 
{\it $\alpha$-H\"older}, if for every $a\in I$ there is $H_a\ge 0$ such that
$\gamma$ is $(H_a,\alpha)$-H\"older at $a\in I$.
\end{definition}

If $\aaa=1$ and we have a uniform bound on $H_{a}$ then we get back the definition of Lipschitz curves.

It is easy to verify the following lemma:

\begin{lemma}\label{alphameasure}
Let $\gamma\colon I\to\R^d$ be an $(H,\alpha)$-H\"older curve for some 
$0<\alpha\le 1$ and $H\ge 0$. Assume that $A\subset I$. Then
$\H^{\frac 1\alpha}(\gamma(A))\le H^{\frac 1\alpha}\H^1(A)$.
\end{lemma}


Since we work with countable covers it is not a problem that 
 in our definition of an $\alpha$-H\"older curve we did not assume that we have a global and uniform constant $H$. In case one wants to work with curves with uniform bound on $H_{a}$ the following lemma can be used:

\begin{lemma}[Lemma 6.3 of \cite{Buirr}]\label{holdercover}
Suppose that $0<\alpha\le 1$. If $\gamma\colon I\to\R^d$ is an $\alpha$-H\"older curve then there is a countable collection of curves
$\gamma_i\colon I\to\R^d$, $i\in\N$, such that $\gamma_i$ is
$(i,\alpha)$-H\"older and 
\[
\gamma(I)\subset\bigcup_{i=1}^\infty\gamma_i(I).
\]
\end{lemma}

The proof of this lemma can be regarded to be an exercise, the details can be found in \cite{Buirr}.

Based on Lemmas \ref{alphameasure} and \ref{holdercover} we define the following generalization of $1$-rectifiability and pure 1-unrectifiability.

\begin{definition}\label{alpharectifiable}
Let $0<\alpha\le 1$. A set $A\subset\R^d$ is {\it $\alpha$-rectifiable}, if
there exist $\alpha$-H\"older curves $\gamma_i\colon I\to\R^d$, $i\in\N$, such
that $\H^{\frac 1\alpha}(A\setminus(\bigcup_{i=1}^\infty\gamma_i(I)))=0$.

 A set 
$A\subset\R^d$ is {\it purely $\alpha$-unrectifiable}, if 
$\H^{\frac 1\alpha}(A\cap\gamma(I))=0$ for all $\alpha$-H\"older curves 
$\gamma\colon I\to\R^d$.
\end{definition}

The above Definition \ref{alpharectifiable} is much less known than the concept of the usual  rectifiability. However,   
 fractional rectifiability questions have already 
been studied by some authors. See, for example the papers \cite{MM1,MM2,MM3} by Mart\'{\i}n and Mattila and results for measures by 
 Badger, Naples and
Vellis in \cite{BadNapVel,BadVell}.

\section{Fractal percolation}\label{model}

\subsection{Definition of the fractal percolation.} \label{deffracperc} 
\begin{figure}[hb]
\centering
\includegraphics[width=0.45\linewidth]{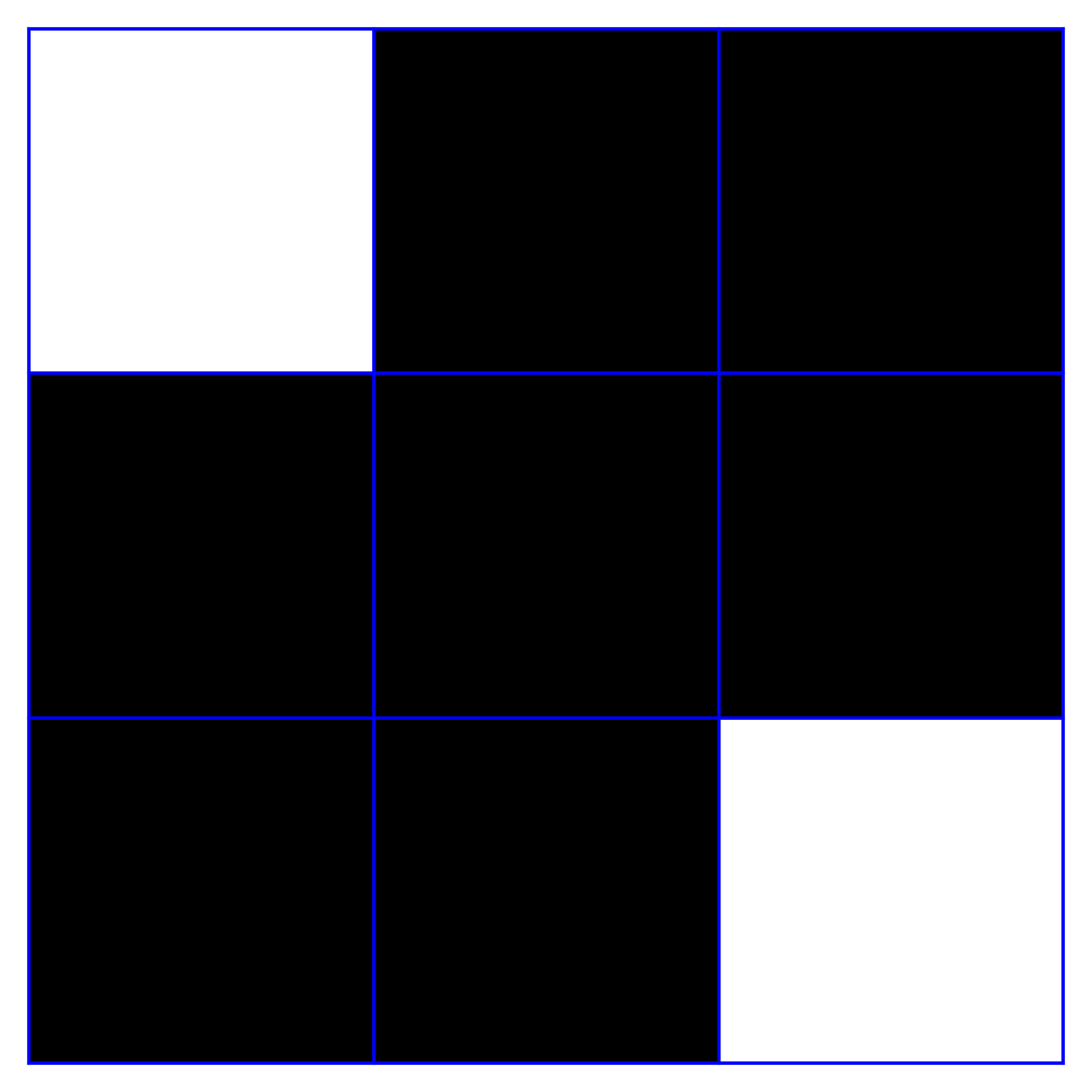}
\includegraphics[width=0.45\linewidth]{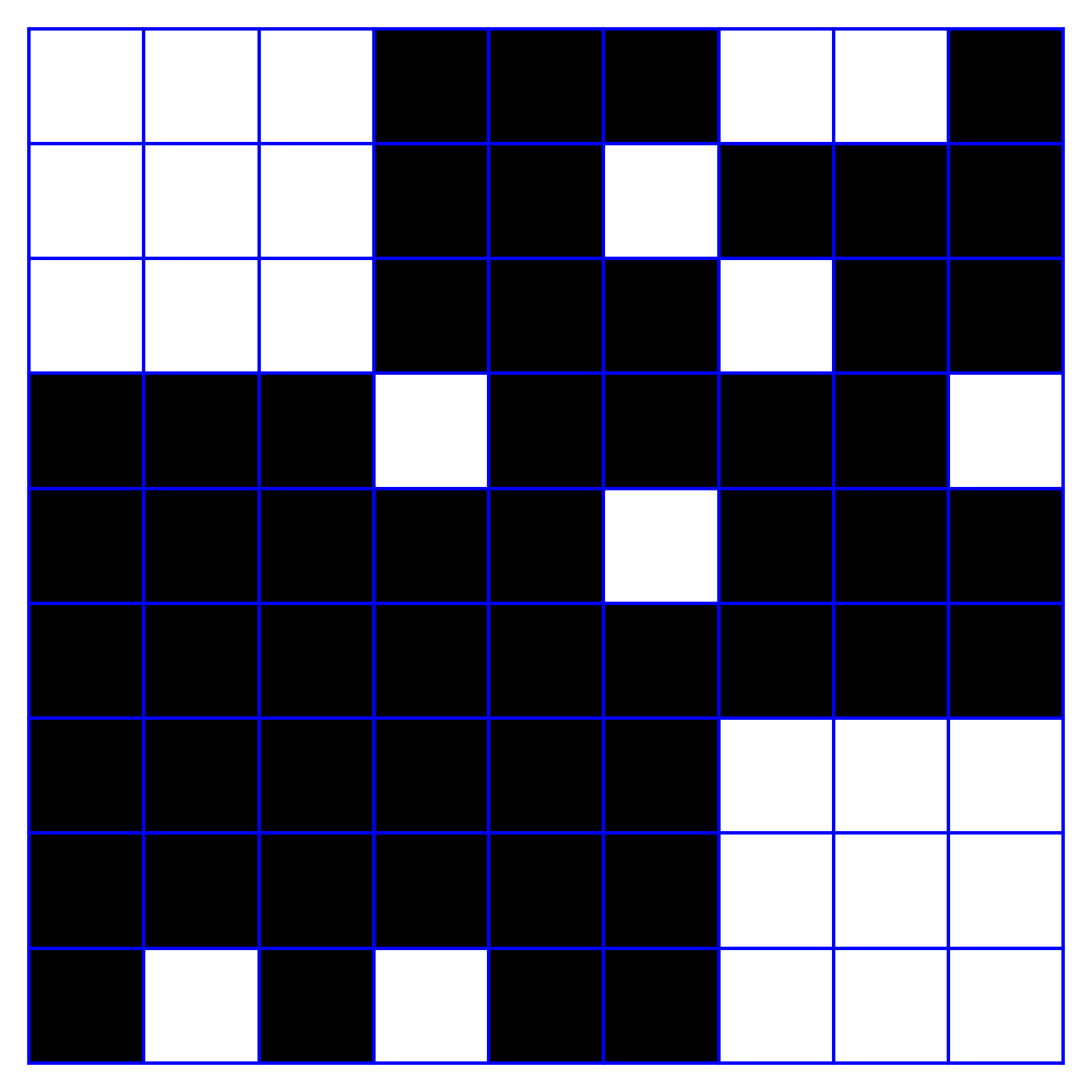}
\captionsetup{width=0.9\textwidth}
\caption{Fractal percolation: $N=3$, Steps 1 and 2, deleted squares are white $p=1-1/7$}
\label{fig-perc1}
\end{figure}

Suppose $0\le p\le 1$, and let $N\in\N$ with $N\ge 2$.
Fixing $d\in\N$, we construct a random compact subset $E$ in the
unit cube $Q_0:=[0,1]^d\subset\R^d$.
First we divide $Q_0$ into $N^d$ subcubes of equal size. Independently of 
each other, we keep/choose one with probability $p$ and delete with
probability $1-p$. We denote by $\C_1$ the collection of the retained subcubes. We repeat independently the same process in each $Q\in\C_1$. 
We denote by $\C_2$ the kept
 cubes at the second level. By  iterating
this process we obtain the sets $\C_{n}$, $n\in \N$. The fractal percolation set $E$, is defined by
\[
E:=\bigcap_{n=1}^\infty \bigcup_{Q\in\C_n}Q.
\]
One possible outcome of this fractal percolation process is illustrated on  
 \cref{fig-perc1,fig-perc2}. This is the planar case $d=2$ with $N=3$, 
 when the percolation was calculated the probability of deleting was 1/7. 
 On the figures the deleted squares are white and the kept ones are black.
 
  In \cite{Buirr} we used the following formal approach, slightly different from the above one. 
For $n\in\{0,1,2,...\}$, let
\[
\Q_n:=\Big\{\prod_{i=1}^d[(l_i-1)N^{-n},l_iN^{-n}]\mid l_i=1,\dots,N^n \text{ and } 
        i=1,\dots,d\Big\}
\]
be the collection of grid subcubes of $Q_0:=[0,1]^d$ with side length $N^{-n}$.
{\it The level of a cube} $Q\in\Q_n$ is $n$.
In our definition at each level we make a random independent decision whether a cube is being kept, or deleted and in the end we take an intersection of the kept regions.
 
Let 
$J:=\{1,\dots,N^d\}$. 
The $n$ long finite words with letters in $J$
are denoted by $J^{n}$.
 
One can define in a natural way a consistent indexing of $\Q_{n}$ by elements of $J^n$ in  a way that if $Q'\in \Q_{n+1}$ is a subcube of $Q\in \Q_{n}$ then the string in $J^{n}$ corresponding to $Q$ is a substring of the string corresponding to $Q'$.
Under this indexing/bijection $Q_\ii$
corresponds to $\ii\in J^n$.

The probability space $\Omega$ is the space of all constructions and
the natural probability measure on $\Omega$ induced by this procedure is
denoted by $\P$. 

 Let 
$\T$ be the rooted $N^d$-branching tree. 
 The set of vertices of $\T$ is denoted by $v(\T)$.
 One can index 
the vertices of $\T$  by
finite words with letters in $J$ in a natural way  by elements of 
$\bigcup_{n=0}^\infty J^n$.
Elements of $J^n$ are used to encode
the vertices whose distance to the root is $n$.
This indexing defines a bijection between vertices of $\T$ and cubes in $\cup_n\Q_{n}$.
For simplicity we identify the vertices of $\T$ with our index set, and we will speak of a vertex $\ii$ instead of writing $v(\ii).$

We put
$\Omega:=\{0,1\}^{v(\T)}=\{\omega\mid\omega\colon v(\T)\to\{0,1\}\}$.
The root $\T$ corresponds to the empty word $0$, and the corresponding cube is the unit cube. 

The elements $\ooo\in\OOO$ define the percolations. Given a vertex $v=\ii$  the cube $Q_\ii$ at level $n$ is kept if $\ooo(\ii)=1$ and deleted if $\ooo(\ii)=0$. 

Given $\omega\in\Omega$  the fractal percolation set $E(\omega)\subset\R^d$ is defined 
 as follows:   the set of kept cubes
in $\Q_n$ is denoted by $\C_n(\omega):=\{Q_\ii\in\Q_n\mid\omega(\ii)=1\}$. For 
every $\omega\in\Omega$, we define the fractal percolation set $E(\omega)$ by
\[
E(\omega):=\bigcap_{n=0}^\infty\bigcup_{Q\in\C_n(\omega)}Q.
\]
Several of our later statements are concerning percolations conditioned on non-extinction, that is those cases when $E(\ooo)\not=\ess.$  Obviously, this holds  if and only if there exists an infinite 
path $T\subset\T$ rooted at $0$ such that $\omega(\ii)=1$ for all vertices
$\ii$ of $T$. It is clear that $E=\emptyset$ with positive probability if $p<1$, since even at the first step  it may happen that all cubes are deleted and therefore
$\C_1$ can be empty. 

In this approach the sets $\bigcup_{Q\in\C_n(\omega)}Q$ are not nested, but they are independent for different $n$s and several technical issues in the proofs have a simpler treatment. If one wants to think of the traditional approach to fractal percolation then the nested sets $\widetilde{C}_n=\cap_{k=0}^n\cup_{Q\in\C_k(\omega)}Q$ can be considered.

 Suppose $\delta_k$ is the Dirac measure at $k\in \{0,1\}$.
Given $0\le p\le 1$, one can define a Borel probability measure $\P$ on $\Omega$ 
by 
\[
\P:=((1-p)\delta_0+p\delta_1)^{v(\T)}.
\]
The probability space for our fractal percolation will be
$(\Omega,\mathcal B,\P)$, where $\mathcal B$ is the completion of the Borel 
$\sigma$-algebra.

\begin{figure}[ht!]
\centering
\includegraphics[width=0.45\linewidth]{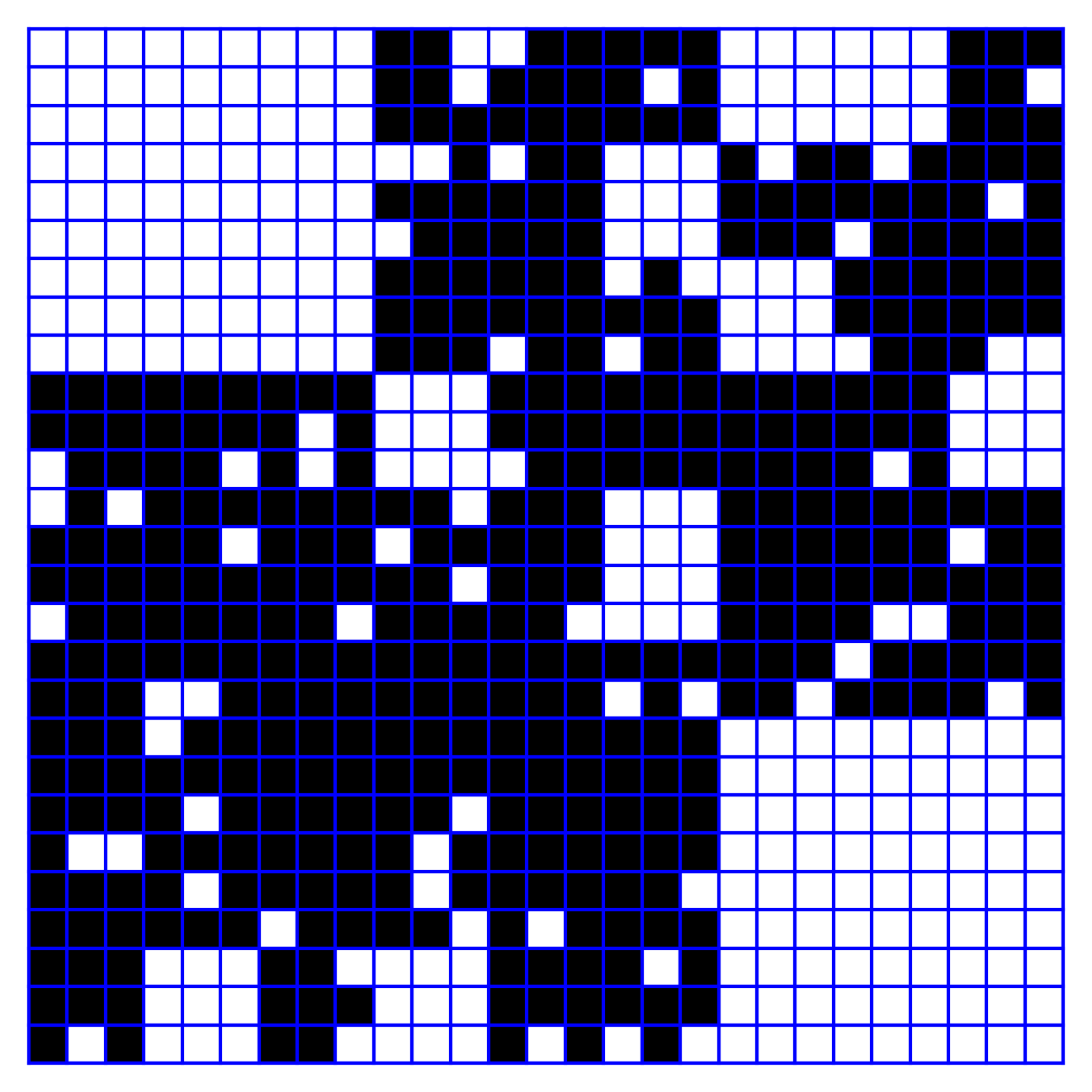}
\includegraphics[width=0.45\linewidth]{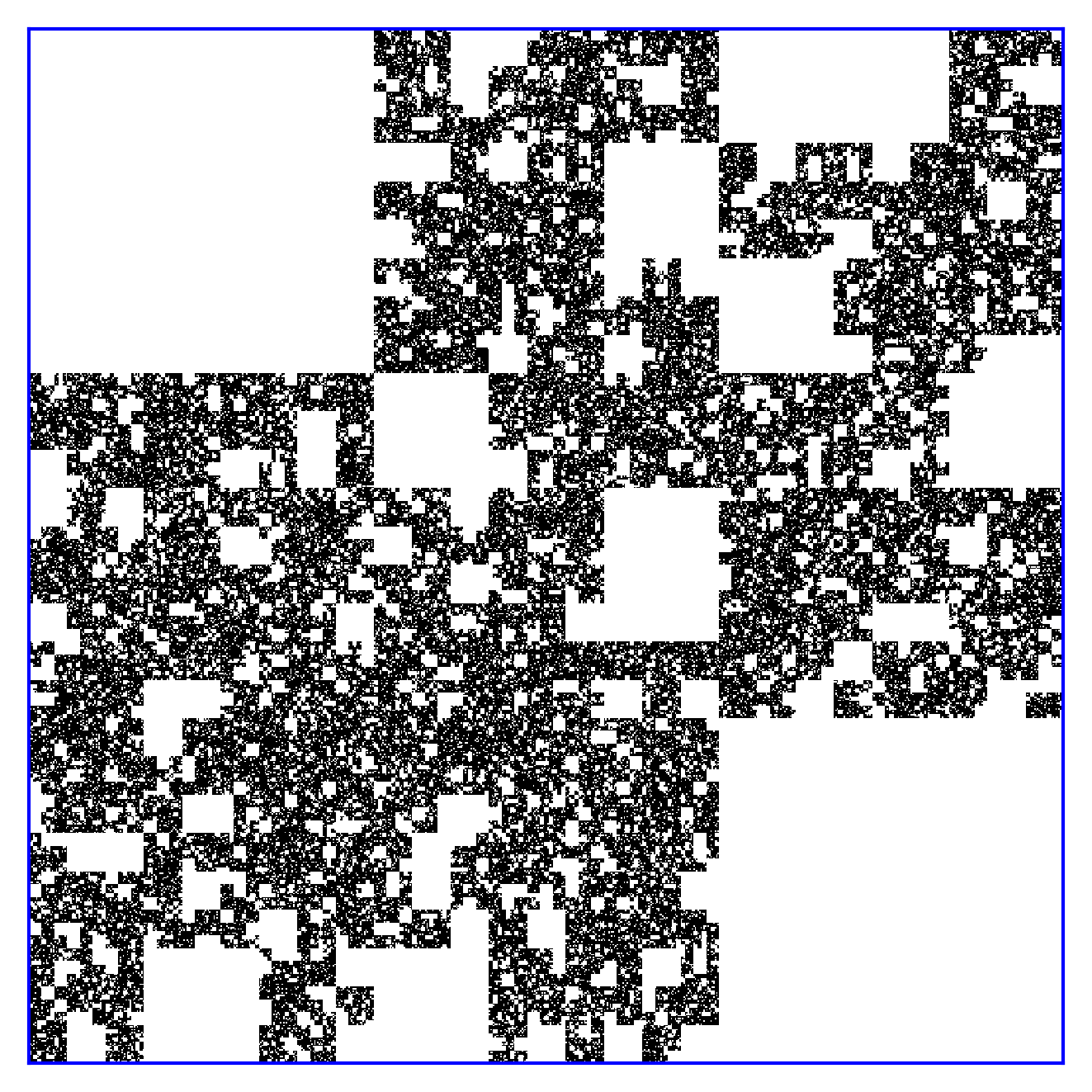}
\captionsetup{width=0.9\textwidth}
\caption{Fractal percolation: $N=3$, Steps 3 and 7, $p=1-1/7$, Step 7 is without grid}
\label{fig-perc2}
\end{figure}

There is a vast literature about fractal percolation. For some further information we refer to  \cite{Chaa} and \cite{Grimpercol}, or to the recent survey \cite{MR4911843}.

The name fractal percolation comes from the fact that first connectedness questions were studied of the set $E$, which is now called the percolation set. The planar, two-dimensional case was considered by  J.T. Chayes, L. Chayes and R. Durrett
 in \cite{CCDur}. They showed that there is
a critical probability $0<p_c<1$ such that if $p<p_c$, then $E$ is totally 
disconnected with probability one, this also includes the cases when $E$ is empty. 

On the other hand, 
$p\ge p_c$ the opposing sides of $Q_0$ are 
connected by a connected component of $E$ with positive probability.  

The exact value of $p_c$ is not known.
In \cite{Don15} for the case shown on our figures, that is for the planar case with $3\times 3$ subdivisions of the square
the estimate is $0.784< p_{c} < 0.940$.  For the creation of the figures we used $p=1-1/7\approx 0.857.$
 There is also a survey preprint \cite{KolTro2025} for further information.

For the higher dimensional cases, $d> 2$ analogous results are true,
see for example \cite{FalGrim}.

There are several further results about the  connectedness properties of $E$. If $p\ge p_c$ Meester in \cite{Mee} showed that under the usual assumptions $E$ contains a non-trivial path connected component.

The results  in \cite{Cha2} imply that  $E$ does not contain uniform $\alpha$-H\"older
curves for $\alpha$ close to $1$ and hence, $E$ does not contain any
rectifiable curve.

In the planar case when $p\ge p_c$
Broman, Camia, Joosten, and  Meester in \cite{BroCamMat} showed that   
there is a decomposition of $E$ into a  totally disconnected part $E^{d}$ and a part $E^c$ which consists of non-trivial connected components
of $E$. 
They also showed that $\dimh E^c<\dimh E^d=\dimh E$ and depending on the parameters there exists $0<\beta<1$, 
 such that $E^c$ is an uncountable union of 
non-trivial $\beta$-H\"older curves. 

Our main result in \cite{Buirr} is the following theorem:

 \begin{theorem}[Theorem 6.13 of \cite{Buirr}]\label{maingeneral}
For all $0\le p<1$, there exists $\alpha_0=\alpha_0(p,d,N)<1$ such that, for 
$\P$-almost all $\omega\in\Omega$, the set $E(\omega)$ is purely
$\alpha$-unrectifiable for all $\alpha_0<\alpha\le 1$.
\end{theorem}

\subsection{Pure $1$-unrectifiability.} \label{subsecpureunr} 
It is worth stating the special case of  Theorem \ref{maingeneral}  with $\aaa=1$: 

\begin{corollary}[Theorem 4.17 of \cite{Buirr}]\label{mainstandard}
For all $0\le p<1$, the set $E(\omega)$ is purely $1$-unrectifiable for 
$\P$-almost all $\omega\in\Omega$.
\end{corollary}

Taking into consideration Corollary \ref{1impliesk} we also obtain: 

\begin{corollary}[Corollary 4.16 of \cite{Buirr}]\label{kunrectifiable}
Let $k\in\N$. For all $0\le p<1$, the set $E(\omega)$ is purely 
$k$-unrectifiable for $\P$-almost all $\omega\in\Omega$.
\end{corollary}

The proof of Theorem \ref{maingeneral} in \cite{Buirr} is very complicated with many technical details. Since Corollary \ref{mainstandard} is interesting in itself in \cite{Buirr} we gave a simpler proof of this special case, but even this simpler proof contains many technical details. To make the idea of the proof of Corollary \ref{mainstandard} more transparent here we give a further simplified heuristic argument.

\begin{figure}[ht!]
\centering
\includegraphics[height=0.45\linewidth]{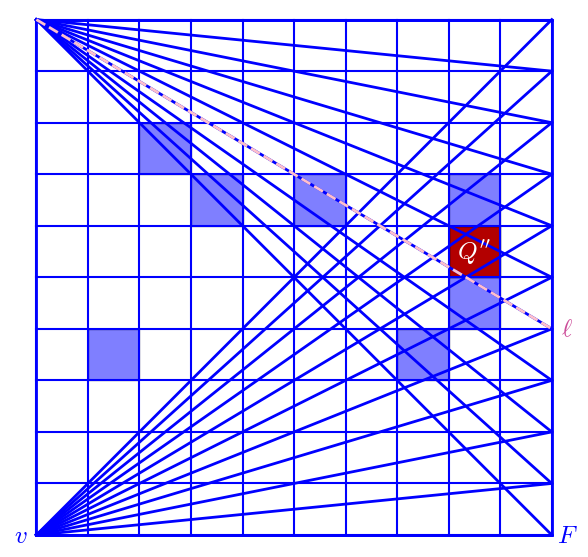}
\includegraphics[height=0.45\linewidth]{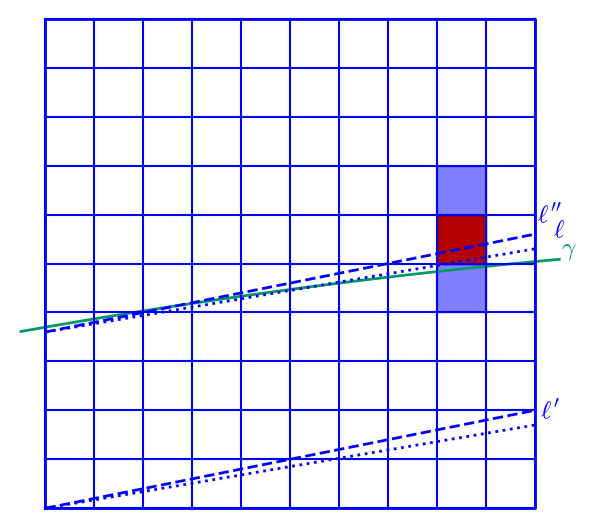}
\captionsetup{width=0.9\textwidth}
\caption{Subdivided squares, $N=10$, $m_{0}=1$, deleted squares are dark}
\label{fig-fpcsquares1}
\end{figure}

\begin{proof}[Simplified heuristic argument for Corollary \ref{mainstandard}]  
We discuss only the planar, $d=2$ case. In this case the $d$-dimensional cubes are squares.
We will often use the abbreviated notation $E$ for the percolation set $E(\ooo)$.
 In this heuristic argument we denote by $c$ a constant which depends only on the probability $p$ and to avoid clumsy indexing from line to line the value of $c$ might change but we keep the same notation, $c$. It may also happen that in the same formula $c$ appears with different ``meaning''. 
 
The key idea is discretization of directions and linear approximation of $C^{1}$ curves.

 We want to prove that conditioned on non-extinction for almost every $\ooo\in\OOO$ 
 the fractal percolation set $E=E(\ooo)$ intersects every $C^{1}$ curve in a set of $\cah^{1}$-measure $0$.

For the proof a sufficiently large $m_{0}$ is chosen and fixed. We suppose that $Q$ is a grid square at a level $m$ that is $Q\in \caq_{m}$. We consider the subgrid squares $Q'\in \caq_{m+m_0}$ at level $m+m_0$.
 From this point on Figure \ref{fig-fpcsquares1} can be useful. On both sides the subdivided square $Q$ is shown. 
The deleted squares are in dark color, blue or red.
On both halves of Figure \ref{fig-fpcsquares1} there is a configuration of three deleted squares vertically on top of each other and the middle square is colored by red. 
We will refer to these grid squares as the red squares and denote the one on the left half of the figure by $Q''$.
 On black and white printed versions of this paper the red squares are darker, in the text we refer to these darker squares still as ``red''.
On the left half of the figure there is a pink dashed line denoted by $\ell$. 
Its {\it principal direction} is horizontal, which means that the angle it makes with a horizontal line is no more than the angle it makes with a vertical line. The principal direction is not necessarily unique, for example the principal direction of the diagonal of $Q$ can be simultaneously vertical and horizontal.
In this proof horizontal ($x$-axis) principal direction is referred to as direction $1$,
vertical ($y$-axis) direction is $2$.
 
On the left side of Figure \ref{fig-fpcsquares1}   the line $\ell$ intersects the red square of side length $1/N^{-m-m_0}$ but the length of the part of $\ell$ in the red square, $Q''$ is a small portion of $1/N^{-m-m_{0}}$. Since the principal direction of $\ell$ is horizontal we consider $\Pi_{1}(\ell\cap Q'')$, where $\Pi_{1}$ denotes the projection onto the $x$-axis (and $\Pi_{2}$ is the projection onto the $y$-axis). We say that a grid square $S\in \caq_{m+m_0}$ is {\it intersected properly }  by a line $\ell'$ with principal direction $i$ if $\Pi_{i}(\ell'\cap S)> {2^{-1}N^{-m-m_{0}}} $. 
In \cite{Buirr} a stronger property ``intersected very properly'' is also introduced, but for this heuristic argument we skip this definition.
The line $\ell$ is not intersecting properly the red square, but it does the blue square under the red square. If a line $\ell$ crosses $Q\in \caq_{m}$ horizontally and properly intersects a deleted sub-grid square at level $\caq_{m+m_0}$ then this deleted sub-grid square removes a portion of $c \cdot N^{-m_{0}}$ of $\ell\cap Q$. If this happens repeatedly at ``all'' (or at ``most'') deeper zoom levels then $\cah^{1}(\ell\cap E)=0$, since again and again a certain fixed portion of the interval $\ell \cap Q$ belongs to deleted squares. 

We need to verify that $\cah^{1}(E\cap \ggg)=0$ for all $C^{1}$ curves $\ggg$. A $C^{1}$ curve $\ggg$ at a sufficiently deep zoom level looks like a line. On the right half of Figure \ref{fig-fpcsquares1}  one can see an ``almost linear'' portion of a curve $\ggg$ crossing the square $Q$. It can be approximated well in the Hausdorff metric by the line $\ell$. 
Of course, due to approximation it may happen that the red square removes a certain portion of $\ell$ but the closely approximating curve $\gamma$ intersects only the nearby blue square,  the one under the red one. If we are lucky, that square is also deleted. This motivates the following definition. A square in $Q$ is strongly $1$-deleted if the squares above and below are also deleted. So on both halves of the figure the red squares are strongly $1$-deleted, but these are the only ones. Considering horizontal neighbors  one can define strongly $2$-deleted squares as well.
The probability that a square is deleted is $1-p$, the probability that it is strongly $1$-(or $2$)-deleted is at least $(1-p)^{3}$ since the deletion of grid squares at a certain zoom level is an independent event.

We do not know in advance the direction of the line $\ell$ which approximates $\ggg$ so we have to take care all possible directions simultaneously.
We want that any line with principal direction $i$, properly intersecting $Q$ hits a strongly $i$-deleted cube. 

We say that  $Q$  is
{\it $m_0$-good for the line $\ell$} if there is a strongly $i$-deleted cube
$Q'\in\Q_{n+m_0}(Q)$ so that
$\ell$ intersects $Q'$ properly and is ``not too close'' (details of this are in \cite{Buirr}) to the sides of the square $Q$ perpendicular to the direction $i$.

Discretization of directions means that we take the four vertices of $Q$. For example we take the bottom left vertex, denoted by $v$ on the left side of  Figure \ref{fig-fpcsquares1}, and we connect it to the grid points on the opposing side of $Q$. 
This side is denoted by $F$ on Figure \ref{fig-fpcsquares1}.
The set of line segments obtained this way is denoted by 
$\Gamma_{F,v}(m_0)$, and equals the discretized set of  line segments with principal direction $1$ going through $v$ on Figure \ref{fig-fpcsquares1}. We take the union of these
line segments 
\[
\Gamma(Q,m_0):=\bigcup_{F}\bigcup_{v}\Gamma_{F,v}(m_0),
\]
where the first union is over all sides (faces) $F$ of $Q$, the 
second one is over all vertices of $Q$ not contained in $F$. On the figure we show both systems of lines corresponding to the face $F$.

We call
a square $Q\in\Q_n$ {\it $m_0$-good}
if it is $m_0$-good for all lines which intersect $Q$ properly
and are parallel to some line in $\Gamma(Q,m_0)$. Finally, a square $Q\in\Q_m$ is
{\it $m_0$-bad} if it is not $m_0$-good.

In our two-dimensional case we have on each side $N^{m_0}+1$ many grid points, $4$ sides and for each side $2$ opposing vertices. That is we have a collection of $8 \cdot (N^{m_0}+1)\leq c \cdot N^{m_{0}}$ many lines in $\Gamma(Q,m_0)$. 

  Assume that a line $\ell$ 
with  a  principle direction $i\in\{1,2\}$
intersects $Q$ properly. 

Recalling that for every $Q'\in\Q_{m+m_0}$, 
\[
\P(\{\omega\in\Omega : Q'\text{ is strongly }i\text{-deleted}\})
  \ge (1-p)^{3},
\]
we obtain that
\[
\P(\{\omega\in\Omega: Q'
\text{ is not strongly }
 i\text{-deleted}\})
  \le 1-(1-p)^{3}=:q.
  \]

\begin{figure}[ht!]
\centering
\includegraphics[width=0.45\linewidth]{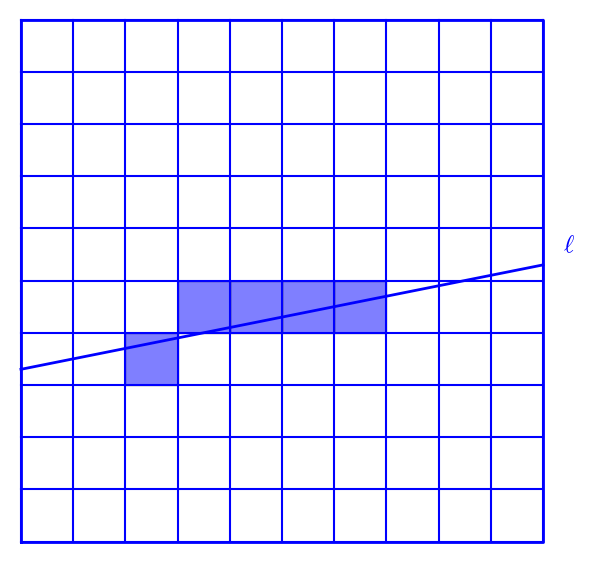}
\includegraphics[width=0.4\linewidth]{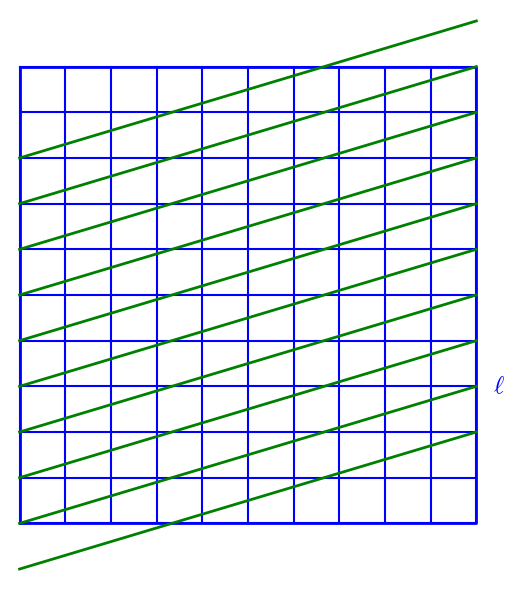}
\captionsetup{width=0.9\textwidth}
\caption{Left: $\K_{m_0}^Q(\ell)$, Right: discretized translated copies of $\ell$.}
\label{fig-skqm}
\end{figure} 

We denote by $\K_{m_0}^Q(\ell)$ a system of 
$m+m_{0}$ level subcubes  $\widetilde{Q}$ of $Q$ such that each $\widetilde{Q}$ is (very) properly intersected by $\ell$, see the left side of Figure \ref{fig-skqm}. Elements of $\K_{m_0}^Q(\ell)$ are in different vertical, or horizontal layers depending on the principal direction of $\ell$ and the number of elements of $\K_{m_0}^Q(\ell)=c \cdot N^{m_{0}}$.
By using independence of deletion properties of squares in different horizontal or vertical layers we obtain that
\[
\P(\{\omega\in\Omega : \text{ every }Q'\in\K_{m_0}^Q(\ell)\text{ is not strongly }
 i\text{-deleted}\})\le q^{c \cdot N^{m_0}}.
\]
Note that $Q$ is $m_0$-good for $\ell$, if at least one of the cubes
$Q'\in\K_{m_0}^Q(\ell)$ is strongly $i$-deleted. Thus,
\[
\P(\{\omega\in\Omega : Q\text{ is }m_0\text{-good for }\ell\})
  \ge 1-q^{c \cdot  N^{m_0}}.
\]

 Now we take a fixed line  $\ell$, from $\Gamma_{F,v}(m_0)$ and consider its translates which properly intersect $Q$. Again we discretize and we regard
translated copies of $\ell$ by a vertical vector of length $N^{-m-m_{0}}$, the side length of one small grid square on the figure, see the right side of Figure \ref{fig-skqm}. We have no more than $c \cdot N^{m_{0}}$ such translates which properly intersect $Q$.  

\begin{proposition}[Proposition 4.9 of \cite{Buirr}]\label{probgood}
For all $0\le p<1$ and $m_0\in\N$, there exists a number
$0\le p_g=p_g(m_0,p,N)\le 1$ with
$\lim_{m_0\to\infty}p_g(m_0,p,N)=1$ such that 
\[
\P(\{\omega\in\Omega: Q\text{ is }m_0\text{-good}\})\ge p_g
\]
for all $n\in\mathbb N$ and $Q\in\Q_m$.
\end{proposition}

This proposition is the exact statement of the fact that if $m_{0}$ is large then it is very likely that a line hits a strongly deleted square. Of course, the larger $m_{0}$, the proportion removed from the line segment is smaller. However, if it repeats infinitely often then using the fact
that $\eta^{n}\to 0$ as $n\to\oo$ 
 for any $0<\eta<1$  one can still obtain that $E$ intersects a line, or a curve well approximated by a line, in a set of one-dimensional Hausdorff measure zero.

\begin{proof}[Proof sketch]
If $Q$ is not $m_0$-good, then there are $i\in\{1,\dots,d\}$, a face $F$ of $Q$
perpendicular to the $i$-th coordinate axis, a vertex $v$ of $Q$ not contained
in $F$ and a line $\ell$ 
with  a  principle direction $i$
intersecting $Q$ properly 
such that
$\ell$ is parallel to some $\ell'\in\Gamma_{F,v}(m_0)$
and $Q$ is $m_0$-bad for $\ell$. In particular, every $Q'\in\K_{m_0}^Q(\ell)$
is not strongly 
 $i$-deleted. The number of different
collections $\K_{m_0}^Q(\ell_b)$ is at most 
$N^{c \cdot m_0}$.
Recalling that we have already seen that
$
\#\Gamma(Q,m_0)\le  c \cdot N^{m_0},
$
 we obtain that
\begin{align*}
\P(\{\omega\in\Omega : Q \text{ is }m_0\text{-bad}\})
&\le c \cdot N^{c \cdot m_0}
q^{c \cdot N^{m_0}}
=:s\to 0, \text{ as }m_{0}\to\oo
\end{align*}
the claim follows with $p_g:=\max\{1-s,0\}$.
\end{proof}

It would be nice if our arbitrary $C^{1}$ curve would intersect at any zoom level only $m_{0}$-good squares. Since with positive probability a square can be $m_{0}$-bad it is almost surely impossible. But we can use the above fact that with a probability $p_{g}$ very close to $1$ we have $m_{0}$-good squares. If this probability is very close to $1$ then the Hausdorff dimension is less than one for the set of those points which are not at infinitely many zoom levels in $m_{0}$-good squares. (In  \cite{Buirr} for some technical reasons we also need to assume that the neighbors of these squares are  also $m_{0}$-good, but by this technical detail we will not want to make this simplified argument more complicated.)
Using the part of  $\ggg$ which is not in the exceptional set of Hausdorff dimension less than one it can be verified that $\cah^{1}(\ggg\cap E)=0$.
\end{proof} 

 \subsection{The general H\"older-$\aaa$ case} \label{secghold} 
Theorem \ref{maingeneral} is the result which we discuss in this section. Its proof in \cite{Buirr} is lengthy with many hard to follow technical details. In this survey paper we omit most of the details and we just outline some ideas behind this proof.
Again we discuss and illustrate the two-dimensional case.
While we zoom in, differentiable curves 
 start  to look almost like lines and, therefore,  the little squares in 
$m_0$-good squares removed during the fractal percolation process  guarantee 
that  differentiable curves cannot intersect the fractal percolation set in 
a set of positive linear measure.
\begin{figure}[ht!]
\centering
\includegraphics[height=0.48\linewidth]{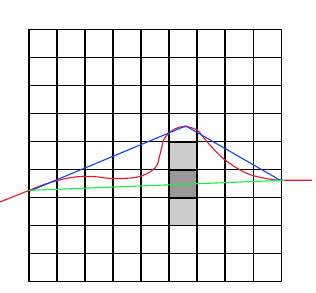}
\includegraphics[height=0.45\linewidth]{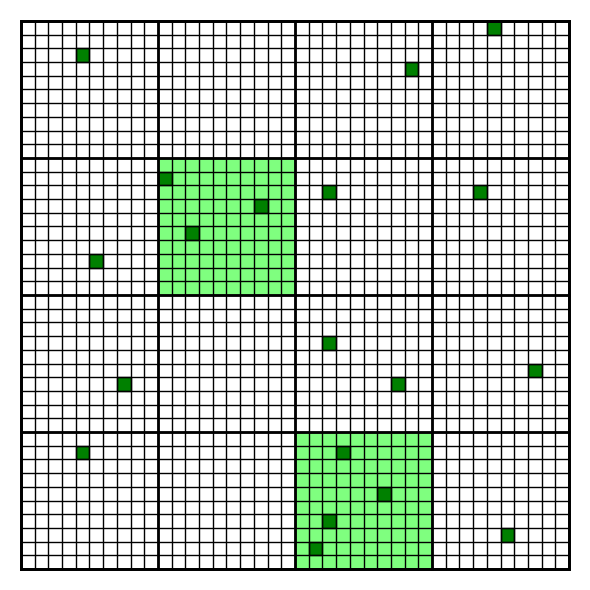}
\captionsetup{width=0.9\textwidth}
\caption{Left: curves going around a deleted square Right: nested good and bad squares, $\mfc=2$.}
\label{fig-fpc39c}
\end{figure} 

Main difficulty of the general Hölder-$\aaa$ case:
 We cannot use tangent lines and linearization. 
We can use broken line approximations. 
If $\gamma$ passes through an 
$m_0$-good square  and is close to the fractal 
percolation set, it has to go around a strongly deleted square, or part of it is removed (see left half of Figure \ref{fig-fpc39c}). Going around  a deleted square increases 
the arc length slightly 
compared to a straight line. If we increase the arc length at many different zoom levels then this might lead to the fact that the curve is not Hölder-$\aaa$ any more. We want that at most zoom levels, most of the points of the curve are in $m_{0}$-good squares. The difficulty one needs to deal with is the fact that at  all zoom levels with positive probability there are $m_{0}$-bad squares. Due to independence
 with high probability these bad squares are sufficiently uniformly spread. 
It is a delicate issue in the proof in \cite{Buirr} to select a correct zooming rate. 
If we zoom out too quickly we might end up with ``larger squares'' which contain too many 
``smaller bad squares'' at a deeper zoom level, while if we zoom out too slowly we might not have a sufficiently quick increase of arc length. 

The key definition in \cite{Buirr} is the 
definition of $(\bL,\mfc,m_0)$-good squares.
 (This is the sequence of different zoom levels we consider.)

Fix $\omega\in\Omega$ and $n\in\N$. Let $\mfc,m_0\in\N$. Assume
that $\bL:=(L_j)_{j=0}^k$ is a strictly decreasing finite sequence of  $k$ many
nonnegative
integers such that $L_k=0$. (This is the sequence of different zoom levels we consider.)

Let $Q\in\Q_n$. We say that $Q$ is
{\it $(\bL,\mfc,m_0)$-good} if it is $(k,\bL,\mfc,m_0)$-good, where the concept
of being $(k,\bL,\mfc,m_0)$-good is defined inductively. If $Q$ is not
$(\bL,\mfc,m_0)$-good, it is {\it $(\bL,\mfc,m_0)$-bad}.

If $k=0$, we say that $Q$ is {\it $(0,\bL,\mfc,m_0)$-good} if it is $m_0$-good.
Otherwise, $Q$ is {\it $(0,\bL,\mfc,m_0)$-bad}. (This is the deepest zoom level. At this level we start with $m_{0}$-good squares.)

Now we skip some technical details of the definition of $(k,\bL,\mfc,m_0)$-good squares.

The goal is that in the end of the inductive definition we have the properety that if $Q\in\Q_n$ is $(k,\bL,\mfc,m_0)$-good, then all but $\mfc$ cubes
$Q'\in\Q_{n+L_{k-1}}(Q)$  are
$(k-1,\bL,\mfc,m_0)$-good. This, in turn, means that all except $\mfc$ cubes
$Q''\in\Q_{n+L_{k-2}}(Q')$ are $(k-2,\bL,\mfc,m_0)$-good, etc.

See the right half of Figure \ref{fig-fpc39c}. The bad squares are shaded in green  and on the figure, a larger square is considered bad if it contains more than two bad smaller squares, so this corresponds to the case $\mfc=2$.

This concludes our quick sketch of ideas behind the proof of Theorem \ref{maingeneral}. 

%
%
%

\section{Regular and irregular sets on the graphs of continuous functions} \label{secregoncont}

Next we will deal with planar irregular and regular $1$-sets obtained by considering graphs of continuous functions. In the first subsection we look at  the Takagi function.

\subsection{ Regular and irregular $1$-sets on the graph of the Takagi function.} \label{subsectak}

\begin{figure}[ht!]
\centering
\includegraphics[width=0.95\linewidth]{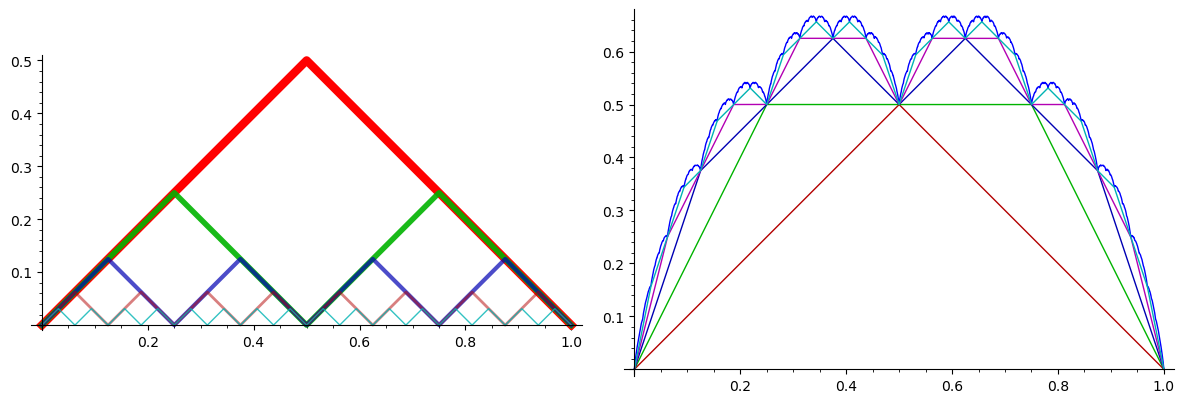}
\captionsetup{width=0.9\textwidth}
\caption{The Takagi function and it partial sums}
\label{fig-taka}
\end{figure}

The Takagi function $T:[0,1]\to \mathbb{R}$ is a classical example of a continuous nowhere differentiable function.

 Let $D$ be the set of dyadic rationals, that is let {$$\ds  
 D_n:=\left \{\frac{k}{2^{n-1}}: k=0,1, \ldots, 2^{n-1}\right \},\quad D:=\bigcup_n D_n. 
$$}
Let $G_n = g_1+\cdots+g_n$ and denote the distance from the point $x$ to the set $D_{n}$ by $ g_n(x)=\text{dist}\, (x,D_n)$. 
The Takagi function is defined as 
\begin{equation}\label{*takdef} 
\ds  
 T(x)=\sum_{n=1}^{\infty} g_n(x)=\lim_{n\to \infty}\, G_n(x).
\end{equation}
On the right side of Figure \ref{fig-taka} one can see the Takagi function and its partial sums 
$n=1,...5,$ on the left side of this figure one can see the functions $g_{n}$ for these $n$s. 

The Takagi function was introduced by {T.Takagi},
\cite{Tak1903} in 1903 as an example 
of an everywhere non-differentiable function on $[0, 1].$  
This function was rediscovered by many people.
In 1930 {van der Waerden} studied in \cite{vdW1930} a decimal version of this function and proved its non-differentiability.
In 1933 {Hildebrandt}, \cite{Hil1933} simplified his construction to rediscover the Takagi function.

In the 1950s {G. de Rham} \cite{dRh1957b,dRh1957a} considered self-similar constructions 
of geometric objects, again constructing a function equivalent to the Takagi function. 
In 1959 {Kahane} \cite{Kah1959} noted that the level set 
$T^{-1}(2/3)$ was a perfect, totally disconnected set of Lebesgue measure $0,$ and in 1984 {Baba} 
\cite{Bab1984}
showed that {$T^{-1}(2/3)$ has Hausdorff dimension $1/2 .$} 
The importance of the level set $T^{-1}(2/3)$ is the fact that $2/3$ is the maximum of the Takagi function. Learning Calculus one got used to the fact that for smooth functions the points with local maxima and minima are usually isolated and the corresponding level sets are finite. For the fractal function $T(x)$ it is just the opposite. In fact, the level set corresponding to $2/3$ is its largest measure level set
this was proved by de Amo, Bhouri,  D\'iaz Carrillo and Fern\'andez-S\'anchez  in \cite{ABCF}. 
 
My study of the Takagi function started with 
the following problem I proposed for the
Mikl\'os Schweitzer Mathematical competition
of the J\'anos Bolyai Mathematical Society in 2006:
 
{\it Suppose that $f(x)=\sum_{n=0}^{{\infty}}2^{-n}||2^{n}x||,$
where $||x||$ is the distance of $x$ from the closest integer
(that is, $f$ is Takagi's function). What can we say
for Lebesgue almost every $y\in f( {\ensuremath {\mathbb R}})$
about
the cardinality of the level set $$L_{y}= \{x\in[0,1]:f(x)=y  \}?$$}

The surprizing answer is that it is finite. 

Level sets of Takagi's function were further studied by several authors see for example the papers \cite{LagMada,LagMadb}  J. Lagarias and Z. Maddock,
and the papers \cite{All2012,MR3016850,MR3122292,MR2869811,
MR3254592,MR3225427} by P. Allaart and K. Kawamura.
See also the recent paper of R. Anttila, B. B\'ar\'any and A. K\"aenm\"aki \cite{AnttilaBaranyKaenmaki2023}.

By Besicovitch's theorem \cite{Besi} on the
projection of irregular $1$-sets
the projection of an irregular $1$-set $S$
in almost all directions is of zero Lebesgue measure, see also
 Theorem 6.13 in \cite{Fal1986}. This latter book contains a very nice introduction to the theory of regular and irregular $1$-sets.

 By the Bolzano-Darboux property  there can be at most one direction in which the projection of
 the graph of a continuous function defined on a non-degenerate interval  does not contain an interval.
 Hence this graph
 cannot be an irregular $1$-set. However, 
quite often one can find in a natural way on the graph of some continuous fractal
functions some irregular $1$-subsets. During the rest of this subsection we will mainly discuss results from \cite{Buirra}. 

The next theorem is about the decomposition of the graph of the Takagi function into regular and irregular $1$-sets.

In this section we will denote by  $\Pi_{x}$ and $\Pi_{y}$ the projections onto the $x$ and $y$ axes.

\begin{theorem}[Part of Theorem 9 in \cite{Buirra}]\label{thtakirrdec} 
There are two disjoint sets  $S_{irr}$ and $S_{reg}$ such that  
 $$S_{irr} \cup S_{reg}= \{(x; {{T}}(x)):x\in [0,1]  \},$$ where $T$ is the Takagi
function.

The set $S_{irr}$
is an irregular $1$-set such that ${\mathcal {L}} (\Pi_{x}(S_{irr}))=1$ and
${\mathcal {L}} (\Pi_{y}(S_{irr}))=0.$

The set
$S_{reg}$ can be covered by the graphs of countably
many monotone increasing functions.
Almost every level set of $T(x)$ is finite.
\end{theorem}

The set $S_{reg}$ is ``responsible" for the finite 
level sets.
The set $S_{irr}$ is { an irregular (or purely unrectifiable) $1$-set.} 
This means that $\cah^{1} ( S_{irr}) $, the $ 1$-dimensional Hausdorff measure of  
$S_{irr}$, is positive and finite, moreover, $S_{irr}$ intersects every continuously  
differentiable curve in a set of $\cah^{1}$-measure zero. By Besicovitch's 
projection theorem  the projection of $S_{irr}$ in almost  
all directions is of zero Lebesgue measure. 
 
\begin{proof}[Proof sketch of Theorem  \ref{thtakirrdec}.] 
Recall the definition of the Takagi function in \eqref{*takdef}. 
Since $ { { G} }_{N}$ is piecewise linear its derivative exists
at all but finitely many points.
If $ { { G} }_{N}'(x)$ exists then {it equals zero} 
if and only if  in the dyadic expansion of $x$ {up to the $(N+1)$st 
digit the number of $0$s equals the number of 
$1$s.} 
It is well-known that for almost every $x\in [0,1]$ 
this happens 
for infinitely many $N$s.

In fact, 
the $ { { G} }'_{N}$s 
 correspond to {a symmetric random walk model,}
where {an $n$th digit $0$ means a unit step in the negative} and {an 
$n$th digit $1$ means a unit step in the positive direction.}
By  P\'olya's theorem, this random walk 
is persistent,
 see for example \cite{Spi76},  that is, 
the particle doing the random walk 
returns with probability one infinitely often to the origin. 
This means zero derivative for us for $G_{N}$ on a dyadic interval.

To give a more precise definition of these points we 
denote by $X_{1}$ the set of those $x\in [0,1]$ for which
$ { { G} }_{N}'(x)$ exists for all $N$ 
and there exists an infinite 
sequence
$N(x,1)<...<N(x,\kappa )<...$ such that 
$ { { G} }_{N(x,\kappa )}'(x)=0$ for all $\kappa \in  { \ensuremath { \mathbb N } }$, but if 
$N\not=N(x,\kappa )$ for all $\kappa $ then $ { { G} }_{N}'(x)\not=0.$ 
 
\begin{figure}[ht!]
\centering
\includegraphics[height=0.32\linewidth]{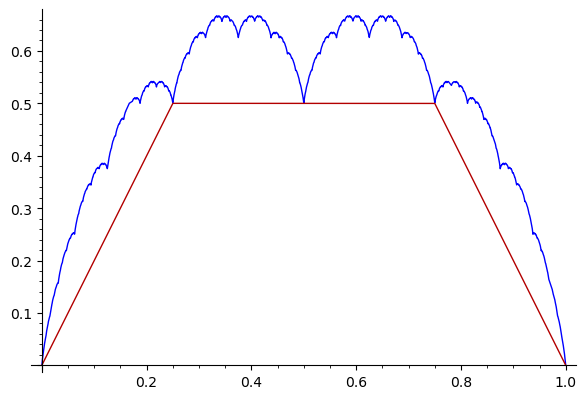}
\includegraphics[width=0.48\linewidth]{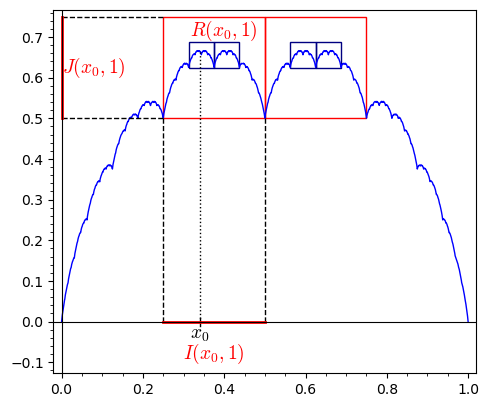} 
\captionsetup{width=0.9\textwidth}
\caption{Left: the Takagi function and $G_{2}$, Right: the square $R(x_0,1)$ on the graph of the Takagi function}
\label{fig-takR}
\end{figure}

 Suppose $ { x_ { 0 } }\in  { X_ { 1 } }$. For any $\kappa =1,2,...$ 
there exists $ { k_ { 0 } }( { x_ { 0 } },\kappa )\in  { \ensuremath { \mathbb Z } }$ such that 
$$\ds%
\label{2*081066} 
 { x_ { 0 } }\in I( { x_ { 0 } },\kappa ) { { \buildrel { \rm def } \over = } } \bigg(\frac{ { k_ { 0 } }( { x_ { 0 } },\kappa )} 
{2^{N( { x_ { 0 } },\kappa )+1}} 
,\frac{ { k_ { 0 } }( { x_ { 0 } },\kappa )+1} 
{2^{N( { x_ { 0 } },\kappa )+1}}\bigg). 
$$%
This means that $I( { x_ { 0 } },\kappa ) $ is a dyadic interval of length $2^{-(N( { x_ { 0 } },\kappa )+1)}$ on which $ { { G} }_{N(x,\kappa )}'(x)=0$.  See Figure \ref{fig-takR}. 
 
We have a rescaling property of the Takagi function, namely 
$$\ds%
\label{3*081069} 
 { { T} }( { x_ { 1 } })= { { G} }_{N( { x_ { 0 } },\kappa )}( { x_ { 0 } })+ 
2^{-N( { x_ { 0 } },\kappa )-1} { { T} } 
\left( 2^{N( { x_ { 0 } },\kappa )+1} 
( { x_ { 1 } }- 
\frac{k( { x_ { 0 } },\kappa )}{2^{N( { x_ { 0 } },\kappa )+1}})\right),
$$%
that is, over the interval $I( { x_ { 0 } },\kappa ) $ there is a scaled down vertically translated copy of the Takagi function.

We define the squares 
{$ 
 R( { x_ { 0 } },\kappa ) { { \buildrel { \rm def } \over = } } 
cl(I( { x_ { 0 } },\kappa )) { \times } J( { x_ { 0 } },\kappa ),  $}
where
 {$$  
 J( { x_ { 0 } },\kappa )= [ { { G} }_{N( { x_ { 0 } },\kappa )}( { x_ { 0 } }), 
 { { G} }_{N( { x_ { 0 } },\kappa )}( { x_ { 0 } }) 
+{\mathcal {L}} (I( { x_ { 0 } },\kappa ))].$$} 
As it can be seen on the right side of Figure \ref{fig-takR}  in these squares the Takagi function is similar to itself.

By elmentary calculations for any $ { x_ { 0 } }\in X_{1}$ 
we have 
$$\ds%
\label{2*R10} 
R( { x_ { 0 } },1) { \subset } [0,1] { \times }[0,1]. 
$$%

Using mathematical induction and rescaling properties 
 one can see that for a fixed $x_{0}$ the squares $R( { x_ { 0 } },\kappa )$ are nested, that is
{$$\ds
R( { x_ { 0 } },1)\supset 
...\supset 
R( { x_ { 0 } },\kappa ) 
\supset R( { x_ { 0 } },\kappa +1)\supset...\, . 
$$}%
The irregular set on the graph of the Takagi function is defined as $$  
 S_{irr}=\bigcap_{\kappa =1}^{ { \infty }}\bigcup_{ { x_ { 0 } }\in X_{1}} 
R( { x_ { 0 } },\kappa ).$$  
Since ${\mathcal {L}} (J( { x_ { 0 } },\kappa ))\to 0$ 
as $\kappa \to { \infty }$, 
the squares $R( { x_ { 0 } },\kappa )$ shrink to one point on the graph of Takagi's function and hence
$S_{irr}$ is a subset of the graph of $ { { T} }$.
 
By the rescaling property the squares $R( { x_ { 0 } },\kappa )$ contain,
$\ds  
  \{ (x; { { T} }(x)):x\in I( { x_ { 0 } },\kappa )  \} ,$
that is they contain the graph of $T$ over the interval $I( { x_ { 0 } },\kappa )$. 
By symmetry properties of the Takagi function 
there are two squares of the form $R(x,1)$ whose projection onto the $y$-axis equals $J(x_{0},1)$ on Figure \ref{fig-takR}. One of them is  $R(x_0,1)$ on  the figure, the other one is its symmetric image with respect to the line $x=1/2$, that is $R(1-x_{0},1)$. One can also see this latter square on the figure. Within the square $R(x_{0},1)$ the small scaled down copy of $T$ is again symmetric. On the figure the smaller square $R(x_{0},2)$ is also shown. This contains the point $(x_{0},T(x_{0}))$. By the symmetry properties now we can find four squares of the form $R(x,2)$ such that $\Pi_{y}(R(x,2))=J(x_{0},2)$ for all of them. All these four smaller squares can be seen on the figure.
Using repeatedly the symmetry properties of the Takagi function we obtain that
$$\label{*081069} 
{\mathcal {L}} ( { { T} }( { X_ { 1 } })) 
\leq {\mathcal {L}} (\Pi_{y}(S_{irr})) 
\leq {\mathcal {L}} (\cup_{ { x_ { 1 } }\in { X_ { 1 } }}J( { x_ { 1 } },\kappa )) 
\leq 2^{-\kappa }{\mathcal {L}} (\cup_{ { x_ { 1 } }\in { X_ { 1 } }}I( { x_ { 1 } },\kappa ))\leq 2^{-\kappa }. 
$$%
  Since this holds for all $\kappa $ we have 
{$$\ds%
\label{2*081069} 
{\mathcal {L}} ( { { T} }( { X_ { 1 } }))={\mathcal {L}} (\Pi_{y}(S_{irr}))=0. 
$$}%
While ${\mathcal {L}}(\Pi_{x}(S_{irr}))=1$ in \cite{Buirra} we verify that given a continuously differentiable curve $\gamma$
we have ${\mathcal {L}}(\Pi_{x}(S_{irr}\cap\ggg))=0$ this  together with $0=\lll(\Pi_{y}(S_{irr}))$ implies that $\cah^{1}(\ggg\cap S_{{irr}})=0$ and hence $S_{irr}$ is indeed an irregular $1$-set.
 
\begin{figure}[ht] 
\includegraphics[width=0.9\linewidth]{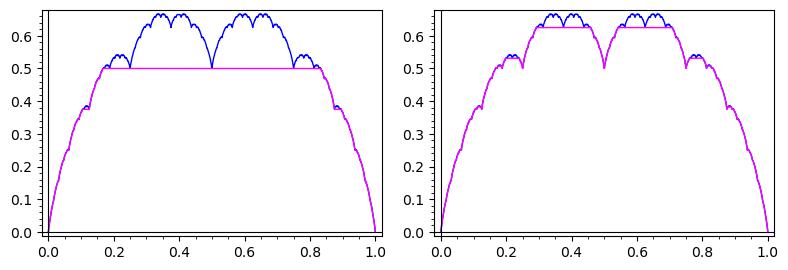} 
\caption{$\mg_1$ and $\mg_2$\label{figg1}} 
\end{figure} 

Next we turn to the regular part of the graph of the Takagi function.
Set {$ 
  \mg_{\kappa }( { x_ { 0 } }) { { \buildrel { \rm def } \over = } }  { { G} }_{N( { x_ { 0 } },\kappa )}( { x_ { 0 } }) \text{ for } 
 { x_ { 0 } }\in  { X_ { 1 } }.$} 
For $x\not \in X_{1}$ extend them so that they are continuous. 
To understand this definition it is useful to look at the functions $\mg_{1}$ and $\mg_{2}$ shown on Figure \ref{figg1}. So $\mg_{1 }( { x_ { 0 } }) { { = } }  { { G} }_{N( { x_ { 0 } },1 )}( { x_ { 0 } }) $ on $I(x_{0},1)\cap X_{1}$. Since ${ { G} }_{N( { x_ { 0 } },1 )}$ is constant on $I(x_{0},1)$ the function $\mg_{1}$ takes this constant value on $I(x_{0},1)$. One can also observe that by construction of the Takagi function if $x_{0}$ is an endpoint of the interval $I(x_{0},1)$ then 
$T(x)={ { G} }_{N( { x_ { 0 } },1 )}( { x_ { 0 } })=\mg_{1}(x_{0})$. This implies that on $[0,1]\sm \cup_{x_{0}\in X_{1}}I(x_{0},1)$ the function $\mg_{1}$ coincides with the Takagi function. 

Observe that $G_{1}(x)$ is monotone increasing on $[0,1/2]$ and decreasing on $[1/2,1]$. For $x_{0}\in X_{1}$  the first index $N$ for $G_{N}(x_{0})$ when 
$G_{N}'(x_{0})=0$ is $N=N(x_{0},1)$. This way $\mg_{1}$ is monotone increasing on $[0,1/2]$ and decreasing on $[1/2,1]$. It is also clear that $\mg_1({X_{1}})$ is a countable set and if $y\not \in \mg_1({X_{1}})$ then $\mg_{1}^{-1}(y)$ consists of at least two points and if $x$ is such a point then $y=\mg_{1}(x)=T(x).$
It is also clear that the Takagi function on each interval of the form $I(x_{0},1)$ coincides with a similar copy of itself. This implies that within each square of the form $R(x_{0},1)$  during the definition of $\mg_{2}$ we get a similar copy of $\mg_{1}$ consisting of two monotone halves and if $x\not\in \cup_{x_{0}\in X_{1}}I(x_{0},2)$ then $\mg_{2}(x)=T(x)$.

{Suppose $\kappa \geq 2$. 
We have 
${\mathcal {L}} (\cup_{ { x_ { 0 } }} J( { x_ { 0 } },\kappa -1))\leq 2^{-\kappa +1}$.} 
By the Borel-Cantelli lemma  
almost every $y$ belongs to only 
 finitely many of the different intervals $J( { x_ { 0 } },\kappa -1)$. 
 
 Suppose 
{$$\ds%
\label{yeq1} 
y_{0}\not \in \bigcap_{\kappa =1}^{ { \infty }}\bigcup_{ { x_ { 1 } }\in { X_ { 1 } }} 
J( { x_ { 1 } },\kappa )=\Pi_{y}(S_{irr}). 
$$}%
Then there exists $\kappa (y_{0})$ such that for all 
$\kappa \geq \kappa (y_{0})$ we have 
$$\ds%
\label{yeq2} 
y_{0}\not \in \bigcup_{ { x_ { 1 } }\in  { X_ { 1 } }} 
J( { x_ { 1 } },\kappa ). 
$$%

Similarity property of the Takagi function scaled down to the squares $R(x_{0},\kappa)$
implies that  
 $\mg_{\kappa } 
( { X_ { 1 } })$ is countable and if 
$y$ belongs to only finitely many 
different $J( { x_ { 0 } },\kappa -1)$ then 
$\mg_{\kappa }^{-1}( \{ y  \})$ 
is finite. 

Hence if $\kappa$ exceeds a suitable constant $\kappa (y_{0})$ 
then $\mg_{\kappa }^{-1}( \{ y_{0}  \})=\mg_{\kappa(y_{0}) }^{-1} 
( \{ y_{0}  \})$. 
For $x\in [0,1] { \setminus } \cup_{ { x_ { 0 } }} I( { x_ { 0 } },\kappa )$ we have $\mg_{\kappa }(x)= 
 { { T} }(x).$
Hence, $ { { T} }^{-1}( \{ y_{0}  \})=\mg_{\kappa(y_{0}) }^{-1} 
( \{ y_{0}  \})$ is finite. 
This implies the statement about the level sets.

 The set {$  
 S_{reg} { { \buildrel { \rm def } \over = } }  \{ (x; { { T} }(x)): x\in [0,1] { \setminus } \Pi_{x}(S_{irr}) 
\}$}  
  is covered by the graphs of the functions $\mg_{\kappa }$. 
   
Therefore $S_{reg}$ can be covered by the graphs of countably 
many monotone functions. 
\end{proof}
 
{{\it Homework:} In the definition of regularity we had  Lipschitz  (or $C^{1}$) curves/maps. Why in the previous proof the cover of $S_{reg}$  by countable many monotone functions is sufficient to show that $S_{reg}$ is a regular $1$-set?}

\begin{remark}\label{defoccmeas} 
Study of level sets  leads naturally to the study of
the occupation measures. Given a function defined on a subinterval of the real line we put 
\begin{equation}\label{defocc} 
\mu (A)=
{\mathcal {L}}  \{x\in[0,1]: {{f}}(x)\in A  \}
={\mathcal {L}} ( {{f}}^{-1}(A)),
 \end{equation}
where we recall that ${\mathcal {L}} $ is the one-dimensional Lebesgue measure.
If we take $f=T$ and  $S_{irr}$ from Theorem  \ref{thtakirrdec}  we see that ${\mathcal {L}}(\Pi_{x}(S_{irr}))=\mu(\Pi_{y}(S_{irr}))=1$
while  $\lll(\Pi_{y}(S_{irr}))=0$, that is the occupation measure of the Takagi function is singular with respect to the Lebesgue measure.
\end{remark}

\begin{figure}[ht] 
\includegraphics[width=0.45\linewidth]{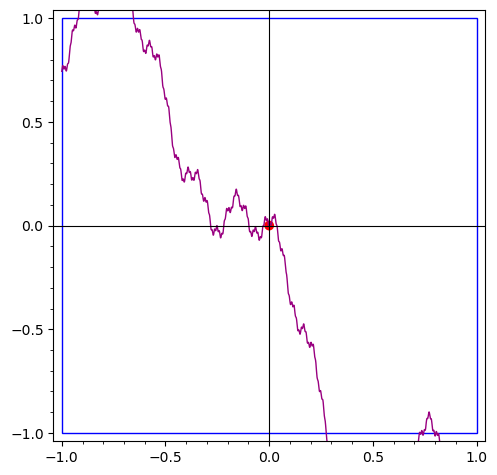} 
\caption{Part of the graph of a function in $Q^{2}$.}\label{figcmt}
\end{figure}

\subsection{Micro tangent sets of continuous functions  and irregular $1$-sets on the graphs of typical continuous functions.} 
\label{subsecmicrocont}
Another natural way one can find regular and irregular $1$-sets on graphs of continuous functions is to consider micro tangent sets. This section is mainly based on papers \cite{Bumicro,BuRa,Buirra}.

We consider graphs of continuous functions in the plane. To avoid some confusion in this section points in $ {\ensuremath {\mathbb R}}^{2}$ will be denoted by $(x;y)$ while the open interval
with endpoints $x$ and $y$ will be denoted by $(x,y)$.

The open ball of radius $r$ centered at $(x;y)$
is denoted by $B((x;y),r)$, while the closed ball is
denoted by $ {{\overline {B}}}((x;y),r).$

The closed cube of side length $2\delta >0$ centered at
$(x;y)$ is denoted by $Q((x;y),\delta )$, that is,
$Q((x;y),\delta )= \{(x';y'): |x'-x|\leq \delta \,  \text{  and  } |y'-y|
\leq \delta   \}$.

For the cube  $Q((0;0),1)$ we use the shorter notation $ {Q^ {2}}$, that is it is
the closed square of side length $2$, centered at
the origin. In the definition of micro tangent sets we will zoom in to a point on the graph of a continuous function and on a $Q^{2}$ ``microscope screen'' we watch the blown up rescaled copy of our function. See Figure 
 \ref{figmicro2} for an illustration. 


The graph of a function $f:[a,b]\to {\ensuremath {\mathbb R}}$ is  denoted by $\graph(f)=\{(x;f(x)):x\in [a,b]\}$.

By $ {C[-1,1]_ {0}}$
we mean the set of those functions  $g$ in $C[-1,1]$ for which $g(0)=0$. We need this subclass, since when we zoom in and blow up a local neighborhood of a point  we get functions passing through the origin.

As it is illustrated on  Figure \ref{figcmt} 
when watching a rescaled enlarged part of the graph of a function it may happen that the graph is not fitting into our $Q^{2}$ ``screen''.
If $F {\subset}  {\ensuremath {\mathbb R}}^2$ is a compact set then we denote by {$CENT(F)$} the connected 
component of $F\cap {Q^ {2}}$ which contains $(0;0)$, this component is the central 
component of $F$ in $ {Q^ {2}}$. On Figure \ref{figcmt}, $CENT(\graph(f)\cap Q^{2})$ is a proper subset of $
\graph(f)\cap Q^{2}$.  

In the next definition we give a precise formulation of our blowing up procedure.

\begin{definition}\label{defMT}

For $\delta _{n}>0$ we put
\begin{equation}\label{ffxodn}
F(f,x_{0},\delta _{n})=\frac1{\delta _{n}} \bigg(
\big (\graph(f)\cap Q((x_{0};f(x_{0})),
\delta _{n})\big)
-(x_{0};f(x_{0}))\bigg),
\end{equation}
that is,
$F(f,x_{0},\delta _{n})$ is the $1/\delta _{n}$-times enlarged part of
$\graph(f)$ belonging to $Q((x_{0};f(x_{0})),\delta _{n})$ translated into
$ {Q^ {2}}$. It is always a compact subset of $Q^{2}$. Working with fractal functions one can obtain many different such compact subsets. See for example the right half of Figure \ref{figmicro2} where the three function graph parts correspond to the three zoom levels shown on the left half of the figure. We are interested in the possible limit points of these compact sets in the Hausdorff metric.

The set $F$ is a {\em micro tangent set ($M$-tangent set)}
of $f$ at $x_{0}$ if there exists $\delta _{n}\searrow 0$ such that
$F(f,x_{0},\delta _{n})$ converges to $F$ in the Hausdorff metric.
The collection of the micro tangent sets of $f$
at $x_{0}$ is denoted by $f_{MT}
(x_{0})$.

The set $F$ is a {\em central-micro tangent set ($CM$-tangent set)}
of $f$ at $x_{0}$ if there exists
$\delta _{n}\searrow 0$ such that $CENT(F(f,x_{0},\delta _{n}))$
converges to $F$ in the Hausdorff metric.
\end{definition}

 \begin{figure}[ht] 
\includegraphics[width=0.9\linewidth]{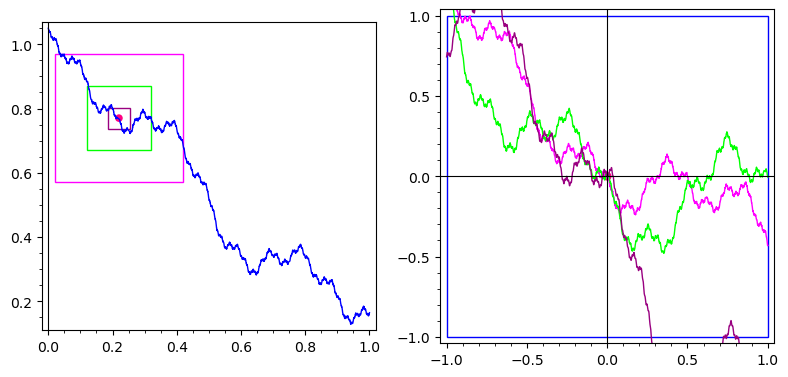} 
\caption{Zooming in on the graph at a point and watching on a  $Q^{2}$ ``screen'' the local neighborhoods}\label{figmicro2}
\end{figure}

The definition of micro tangent sets is closely related to the definition of {\it microsets } by H. Furstenberg \cite{Furmicro}.
Suppose that $A {\subset}  {\ensuremath {\mathbb R}}^{m}$ is compact. The set
$A'$ is a microset of $A$ if there exist sequences of scaling
constants $\gamma_{n}\in {\ensuremath {\mathbb R}}$ and translation vectors $t_{n}\in  {\ensuremath {\mathbb R}}^{m}$
such that $\gamma_{n} A+t_{n}\cap Q^{(m)}$ converges to $A'$ in the
Hausdorff metric.

It is easy to see that if $f$ is differentiable at $x_{0}$ then
$f_{MT}(x_{0})=f_{CMT}(x_{0})$ consists of one line segment
of slope $f'(x_{0})$ passing through the origin.
When $f'(x_{0})=+\oo$  then $f_{MT}(x_{0})$ consists of one vertical line segment passing through $(0;0)$. This motivates the next definition.

 \begin{definition}
We say that $x_{0}$ is {\em  a graph like, or a central graph like} micro tangent point
for $f$ if there exists $g\in  {C[-1,1]_ {0}}$ such that $\graph(g)\cap {Q^ {2}}\in f_{MT}(x_{0
})$, or $\graph(g)\cap {Q^ {2}}\in f_{CMT}(x_{0})$, respectively.
We denote by $GLMT(f)$, or by $CGLMT(f)$
the set of graph like micro tangent points, or the set
of central graph like micro tangent points of $f$, that is,
the set of those $(x_{0};f(x_{0}))$
for which $x_{0}$ is a graph like, or central graph like $MT$-point of
$f$, respectively.
\end{definition}

Clearly, $CGLMT(f)\supset GLMT(f)$. If $f'(x)=+\oo$ then obviously $x$ is not in $CGLMT(f)$.

 \begin{definition}
 For a fixed $g\in {C[-1,1]_ {0}}$
we denote by $GLMT_{g}(f)$, or by $CGLMT_{g}(f)$
the set of those $(x_{0};f(x_{0}))$ for which $\graph(g)\cap {Q^ {2}}$
belongs to $f_{MT}(x_{0})$, or
$CENT(\graph(g)\cap {Q^ {2}})$ belongs to
$f_{CMT}(x_{0})$, respectively.
\end{definition}

Clearly, $GLMT_{g}(f) {\subset} GLMT(f)$, $GLMT_{g}(f) {\subset} CGLMT_{g}(f)$.
Using the fact that we can always find a $g_{1}\in {C[-1,1]_ {0}}$ for which
$\graph(g_{1})\cap {Q^ {2}}=CENT(\graph(g)\cap {Q^ {2}})$ we also have
$CGLMT_{g}(f) {\subset} CGLMT(f)$.

The next definition is the most interesting for us. This describes the wildest possible local behavior.

 \begin{definition}\label{UMTdef}
We say that $x_{0}$ is a {\em universal $MT$-point} for $f$ if for any $g\in {C[-1,1]_ {0}}$ the point
$x_{0}$ is in $GLMT_{g}(f)$.

The collection of those points $(x_{0};f(x_{0}))$ for which
$x_{0}$ is a universal $MT$-point of $f$ will be denoted by
$UMT(f).$
\end{definition}

 \begin{theorem}[Theorem 2 of \cite{Buirra}]\label{cglmtsf}
For any function $f\in C[0,1]$ the sets $GLMT(f)$ and $CGLMT(f)$
are of $\sigma$-finite $ {{\mathcal H} ^ {1}}$-measure.
\end{theorem}

Here is one more theorem valid for any continuous function.

\begin{theorem}[Theorem 11 of \cite{Buirra}]\label{th2}
If $f\in C[0,1]$  then
\begin{equation}\label{**A4}
 {{\cal H} ^ {1}}(UMT(f))\leq 2 .
\end{equation}
\end{theorem}

Given a property we say that the typical continuous function has this property if
the  set, $\caf$ of functions in $C \left [0,1 \right]$ having this property is
residual, that is its complement is of first Baire category which means that it is the countable union of nowhere dense sets. Most often one can prove that $\caf$ is a dense $G_{\ddd}$ set. 
By a result of R. D. Mauldin and S. C. Williams (\cite{MauWil} Theorem 2)
the graph of the typical continuous function is of Hausdorff dimension
one, but is not of $\sigma$-finite $ {{\mathcal H} ^ {1}}$-measure. So, Theorem \ref{cglmtsf}
says that most points in the sense of $ {{\mathcal H} ^ {1}}$-measure
on the graph of the typical continuous function
 are neither graph, nor central graph like.
The next lemma implies that despite this relative
smallness of $GLMT(f)$ and $CGLMT(f)$ the projection
of these sets onto the $x$-axis is of full measure and, as we
will see in Theorem \ref{umt} below,  the projection of $UMT(f)$
onto the $x$-axis is also of full measure.

 \begin{lemma}[Lemma 3 of \cite{Bumicro}]
\label{preumt}
For a fixed $g\in  {C[-1,1]_ {0}}$ the set of those functions $f
\in C[0,1]$ for which $ {{\mathcal {L}}}(\Pi_{x}(GLMT_{g}(f)))=
 {{\mathcal {L}}}(\Pi_{x}(CGLMT_{g}(f)))=1= {{\mathcal {L}}}([0,1])$
is a dense $G_{\delta}$ set in $C[0,1].$

\end{lemma}

Using the fact that $C[-1,1]_ {0}$ is separable and considering a countable dense set $g_{n}$ 
the next theorem follows from Lemma \ref{preumt}. 

 \begin{theorem}[Theorem 5 of \cite{Bumicro}]\label{umt}
There is a dense $G_{\delta}$ set $ {{\mathcal G}}$ of $C[0,1]$ such that
$ {{\mathcal {L}}}(\Pi_{x}(UMT(f)))=1$ for all $f\in  {{\mathcal G}}$.
Furthermore, $UMT(f)$ is a dense $G_{\delta}$ subset in the relative
topology of $\graph(f)$.
Hence, for the typical
continuous function in $C[0,1]$ almost every $x\in [0,1]$ is a universal
micro tangent point and a typical point on the graph of $f$ is in $UMT(f)$.
\end{theorem}

Assume that $g_{0}$ denotes the identically zero function in $[-1,1].$
Then $CGLMT_{g_{0}}(f)=GLMT_{g_{0}}(f)$. It is clear that for a small $\ddd>0$ if a function $g$ is $\delta$ close to $g_{0}$ in the Hausdorff metric then  $\lll(\Pi_y(\graph g))<2\ddd$.
Using the definition of micro tangent sets  one can consider a cover of $UMT(f)$
by squares of the form $ Q((x_{0};f(x_{0})),
\delta)$
such that
$\frac1{\eta } \bigg(
\big (\graph(f)\cap Q((x_{0};f(x_{0})),
\eta )\big)
-(x_{0};f(x_{0}))\bigg),$ equals the graph of a function $g\in C[-1,1]_ {0}$ which is  $\delta$ close to $g_{0}$, using this one can see that $\lll(\Pi_{y}(UMT(f)))<2\ddd$. This argument is behind the next lemma:

 \begin{lemma}[Lemma 6 of \cite{Bumicro}]\label{prepiy2}
We have $ {{\mathcal {L}}}(\Pi_{y}(GLMT_{g_{0}}(f)))=0$ for any $f\in C[0,1].$
\end{lemma}

Since $(x_{0};f(x_{0}))\in UMT(f)$ implies that $(x_{0};f(x_{0}))\in GLMT_{g_{0}}(f))$
Lemma \ref{prepiy2} implies that
$ UMT(f)$ has  small, $\lll$-measure $0$, $y$-projection, but by Theorem  \ref{umt} its $x$-projection is of full measure.

 \begin{theorem}[Theorem 7 of \cite{Bumicro}]\label{piy}
There is a dense $G_{\delta}$ set $ {{\mathcal G}}$ of $C[0,1]$ such
that $ {{\mathcal {L}}}(\Pi_{y}(UMT(f)))=0$ for all $f\in {{\mathcal G}}.$
Hence any preimage of almost every $y$ in the range of
the typical continuous function is not a $UMT$-point.
\end{theorem}

The $UMT(f)$ set of a typical continuous function
is quite stable. 
\begin{theorem}[Theorem 10 of \cite{Buirra}]\label{th1}
Suppose $f\in C[0,1]$ and $h\in C[0,1]$ is differentiable
at $x_{0}\in (0,1)$.
If $(x_{0};f(x_{0}))\in UMT(f)$ then $(x_{0};f(x_{0})+h(x_{0}))
\in UMT(f+h)$.
\end{theorem}

Theorem \ref{th2}
implies that $ {{\cal H} ^ {1}}(UMT(f))\leq 2< {\infty}$,
(in fact, one can verify that $ {{\cal H} ^ {1}}(UMT(f))=1$). Therefore $UMT(f)$ is a $1$-set.

Taking a typical continuous
function $f(x)$ one can apply Theorem \ref{th1}
with $h_{c}(x)=cx$ where $c\in  {\ensuremath {\mathbb R}}$ is an arbitrary constant.
Since $UMT(f+h_{c}) {\subset} GLMT_{g_{0}}(f+h_{c})$
by Lemma \ref{prepiy2}, $ {{\mathcal {L}}}(\pi_{y}(UMT(f+h_{c})))=0$ holds
for any $c\in  {\ensuremath {\mathbb R}}$. 
hence if we denote by $\Pi'_{c}$ the projection onto the line $y=c \cdot x$ then for the typical continuous function $\lll(\Pi'_{c}(UMT(f)))=0$ for all $c\not=0$. Hence on the graph of the typical continuous function $UMT(f)$ is a very interesting irregular $1$-set. Its projection is of zero Lebesgue measure in all directions with the only exception of the projection onto the $x$-axis.
 Hence, $UMT(f)$ for the
typical continuous function is an irregular $1$-set.


Recall the definition of occupation measures in \eqref{defocc}. The argument about the Takagi function following this definition  can be used for typical continuous functions in a very strong way. 
By Theorem \ref{th1}
 if $h\in C[0,1]$ is almost everywhere
differentiable then for almost every $x\in [0,1]$
from $(x;f(x))\in UMT(f)$ it follows that
$(x;f(x)+h(x))\in UMT(f+h)$, which implies
that the occupation measure of $f+h$ is also singular.

In (\cite{HuPe})  P. Humke and G. Petruska
proved that the packing dimension of the typical
continuous function equals two. 
In this respect for the typical continuous function  $UMT(f)$ is sufficiently large.

 \begin{theorem}[Theorem 8 of \cite{Bumicro}]\label{packing}
For the typical continuous function $f\in C[0,1]$ the packing
dimension of $UMT(f)$ equals two.
\end{theorem}

On the other hand, by Theorem \ref{piy} for the typical continuous function\\  $ {{\mathcal {L}}}(\Pi_{y}(UMT(f)))=0$. Hence, in this respect $UMT(f)$
is small. This leads to the natural question about the micro tangent behavior at other points.

\subsection{Micro tangent sets of typical continuous functions   at other, non-$UMT$ points. Vertical universality. } 
\label{subsecmicrocontvert}
In this subsection we discuss some results from  \cite{BuRa}.
Theorems \ref{cglmtsf} and \ref{piy} show that there are not that many universal micro tangent points on the graph of a typical continuous function.  
In fact   for any continuous function the
graph-like, or central graph-like micro tangent points are of $\sigma$-finite
$ {{\mathcal H}}^1$-measure, while as we mentioned after Theorem \ref{cglmtsf} the graph of the typical continuous function is not of $\sigma$-finite
$ {{\mathcal H}}^1$-measure. 

This means that at many points the central part of the micro tangent set system cannot contain the graph of a function. That is  we are interested in points $(x_0;f(x_0))$ where
$f_{MT}(x_0)$ contains sets which consist of vertical line segments, that is,
sets which are of the form $F \times [-1,1]$, where $F {\subset}  [-1,1]$ is an arbitrary
closed set with $0\in F$.

For an arbitrary real function $f$ denote by $L_{\alpha,f}$ 
the horizontal level set at height $ {\alpha}$, that is the set
$\{(x;f(x)):f(x)= {\alpha}\}$.

The accumulation points of $L_{{\alpha},f}$ are denoted by $P_{{\alpha},f}$. Those points which are only right, respectively left accumulation points of $P_{{\alpha},f}$ are denoted by $P^+_{{\alpha},f}$ and $P^-_{{\alpha},f}$, respectively.
Finally, set $P^{\pm}_{{\alpha},f}=P^-_{{\alpha},f}\cup P^+_{{\alpha},f}$.

Next we need some more definitions in order to define vertical universality.

We still  assume that
$[a,b]$ denotes the closed interval with endpoints
$a$ and $b$ (even when $a>b$).

Assume $f\in C{\left[0,1\right]},$ and $r\in \left[0,1\right]$.

Put $P^+_{f}=\bigcup_{\alpha\in {\ensuremath {\mathbb R}}}\Pi_x(P^+_{\alpha,f}),P^-_{f}=\bigcup_{\alpha\in {\ensuremath {\mathbb R}}}\Pi_x(P^-_{\alpha,f}),
P^{\pm}_{f}=\bigcup_{{\alpha}\in {\ensuremath {\mathbb R}}}\Pi_x(P^{\pm}_{{\alpha},f})$ and
$P_{f}=\Pi_x(\bigcup_{{\alpha}\in {\ensuremath {\mathbb R}}}P_{{\alpha},f}\setminus P^{\pm}_{{\alpha},f})$.

We also put
$P^+_{f,r,\leq}=\{x\in\left[r,1\right]:x\in P^+_{f},f(x')\leq f(x),$ $\forall x'\in (r,x)\}$,
and $P^-_{f,r,\leq}=\{x\in\left[0,r\right]:x\in P^-_{f},f(x')\leq f(x),$
$\forall x'\in (x,r)\}$, one can define similarly the sets $P^+_{f,r,\geq},P^-_{f,r,\geq}$.\\
Next observe that
$$P^+_f=\bigcup_{r\in {\ensuremath {\mathbb Q}}\cap\left[0,1\right]}
(P^+_{f,r,\leq}\cup P^+_{f,r,\geq})\text{    and    }
P^-_f=\bigcup_{r\in {\ensuremath {\mathbb Q}}\cap\left[0,1\right]}(P^-_{f,r,\leq}\cup P^-_{f,r,\geq}).$$
It is not difficult to see that $f\mid_{P^+_{f,r,\leq}}$ is monotone increasing.
Hence,
being the graph of a function of bounded variation,
 the $ {{\cal H}}^1$-measure of $f|_{P_{{\alpha},r,\leq}^+}$
is finite. 
 From this one can easily infer that
the set of those
$(x_0;f(x_0))$ which are only one-sided
 accumulation points of $P_{{\alpha},f}$,
 that is the set $\cup_{{\alpha}\in {\ensuremath {\mathbb R}}}P_{{\alpha},f}^{\pm}$ is of
$\sigma $-finite-$ {{\cal H}}^1$ measure.

From Lemma 15 of \cite{BuRa} it follows that $ {{\mathcal {L}}}(P^+_f)=0$ and similarly $ {{\mathcal {L}}}(P^-_f)=0$.

Given $-1\leq x_1<x_2<...<x_ {\nu}\leq 1 ,$ we can consider finite sets of vertical line segments $ V_{x_1,...,x_ {\nu}}=\{(x_j;y)\in Q^2:y\in$
$ \left[-1,1\right],j\in \{1,..., {\nu}\} \}.$

 \begin{definition}\label{D32}
Assume $f\in C\left[0,1\right],x_0\in P^+_f$, that is, letting $ {\alpha}=f(x_0)$ the point
$(x_0;f(x_0))$ is a right-sided accumulation point  of the level set $L_{{\alpha},f}$.
We have {\it right vertical universality} at $x_0\in P^+_f$ if for any choice of
$0= x_1<x_2<...<x_ {\nu}\leq 1 $ we have
\begin{equation}\label{A6}
V_{x_1,...,x_ {\nu}} \in f_{MT}(x_0).
\end{equation}
\end{definition}

One can define similarly the {\it left vertical universality} at $x_0\in P^-_f$.

We have {\it two-sided vertical universality} at $x_0\in P_f$ if for any
$-1\leq x_1<x_2<...<x_ {\nu}\leq 1 $, $0\in\{x_1,...,x_ {\nu} \}$ we have \eqref{A6}.

If we have a closed set $F {\subset} [0,1]$ which contains $0$ then $F\times [-1,1]$ can be approximated in the
Hausdorff-metric by sets of the form $V_{x_1,...,x_ {\nu}} $ with $0=x_1$. Therefore right vertical universality is equivalent to the fact that
 at $x_0$ for any closed set $F {\subset} [0,1]$ containing $0$ we have  $F\times [-1,1]\in f_{MT}(x_0)$.
Similar statements hold for left-, or two-sided vertical universality, with $0\in F$
and $F {\subset} [-1,0]$ or $F {\subset} [-1,1]$, respectively.


 \begin{theorem}[Theorem 18 of \cite{BuRa}]\label{T2}
For the typical continuous function, $f$, for all $ {\alpha} \in  {\ensuremath {\mathbb R}}$ if $P_{{\alpha},f}\ne  {\emptyset}$
then there exists $\widehat{P}_{{\alpha},f} {\subset} P_{{\alpha},f}$, dense $ {{\cal G}}_{{\delta}}$ in $P_{{\alpha},f}$
such that if $x_0\in \pi_x(\widehat{P}_{{\alpha},f})$ then we have two-sided vertical
universality at $x_0$.
\end{theorem}

Denote by $VMT(f)$ the set of those $(x_0;f(x_0))$ for which we have two-sided vertical
universality. The next theorem and Theorem \ref{cglmtsf}  show that for the typical continuous function 
$VMT(f)$ is indeed much larger than $UMT(f)$ or $CGLMT(f)\supset UMT(f)$.

 \begin{theorem}[Theorem 19  of \cite{BuRa}] \label{thvmt} 
For the typical continuous function $VMT(f)$ is of non-$\sigma$-finite $ {{\cal H}}^1$-measure.
\end{theorem}

However, there are still  points on the graph of the typical continuous function which were not discussed so far.
The next theorem is about a $c$-dense set
in the level set where we do not have two-sided universality.

 \begin{theorem}[Theorem 20  of \cite{BuRa}]\label{T3}
For the typical continuous function $f$, for all $ {\alpha}\in  {\ensuremath {\mathbb R}}$,
 if $P_{{\alpha},f}\ne  {\emptyset}$ then
there exists $ {{\widehat {P}}}_{{\alpha},f}' {\subset} P_{{\alpha},f}$, c-dense in $P_{{\alpha},f}$ such that if
$x_0 \in \pi_x( {{\widehat {P}}}_{{\alpha},f}')$ then we do not have two-sided vertical universality at $x_0$.
\end{theorem}

Although Theorem \ref{T3} is about some exceptional points next we see that there are not that many such points. They are not only of zero Hausdorff dimension but they are of measure zero for any given fixed gauge function. 
Recall that  $ {\phi}:[0,+ {\infty})\rightarrow [0,+ {\infty}]$ is called a
Hausdorff gauge function if it is monotone
increasing, continuous from the right and $ {\phi}(x)>0$ for $x>0$. (See
for example \cite{Rog}.)
By results of Kirchheim \cite{Kir}
 for the typical continuous function $L_{{\alpha},f}$ is of zero Hausdorff
dimension. In fact, if $ {\phi}$ is an arbitrary fixed gauge function satisfying
$  \displaystyle  \lim_{x\rightarrow 0+} {\phi}(x)=0$ then $ {{\cal H}}^{{\phi}}(L_{{\alpha},f})=0$. It is worth remarking that it is important to observe that Kirchheim's result and our Theorem \ref{T4} is about one arbitrary, but fixed gauge function $\phi$. A set $S$ is called of strong dimension zero if ${{\cal H}}^{{\phi}}(S)=0$ for all gauge functions $\phi$. It is obvious that Kirchheim's result cannot be strengthened to have strong Hausdorff dimension zero since there are models of ZFC (see Laver \cite{Lav76}) in which Borel's conjecture is true, that is each strong measure zero set is countable, while level sets in the interior of the range of a typical continuous function are uncountable (of cardinality continuum). It is also worth mentioning that Borel's conjecture is independent under ZFC. Using the continuum hypothesis  Sierp\'niski
 \cite{Sier28} proved that there exist uncountable strong measure zero sets. 
 
  Next is the theorem about vertical universality and small exceptional sets.

 \begin{theorem}[Theorem 21  of \cite{BuRa}]\label{T4}
Assume $ {\phi}$ is a Hausdorff gauge function $ {\phi}$ with $  \displaystyle  \lim_{x\rightarrow 0+} {\phi}(x)=0$. For the
typical continuous function, $f\in C[0,1]$ there exists $E_f$ with $ {{\cal H}}^{{\phi}}(E_f)=0$
such that\\
(i) if $ {\alpha} \not \in E_f$ then $f$ has right, or left vertical universality at $(x_0; {\alpha})$ for all
$(x_0; {\alpha})\in P_{{\alpha},f}^+$, or $P_{{\alpha},f}^-$, respectively;\\
(ii) if $ {\alpha}\in E_f$ then there is one exceptional point $(x_0^*; {\alpha})\in P_{{\alpha},f}^+\cup P_{{\alpha},f}^-$
where $f$ does not have right, or left vertical universality and for other
$(x_0; {\alpha})\in P_{{\alpha},f}^+\cup P_{{\alpha},f}^-$ (i) holds.
\end{theorem}

If one considers level sets of typical continuous functions it is worth to say a few words about isolated points on the level sets.
We recall a theorem of  Bruckner and Garg  \cite{BruGa} which states that a  typical
continuous function $f$ in $C \left [0,1 \right]$ satisfies the BG-property.
 Denoting by $min_f$ and $Max_f$ the
minimum and maximum of an $f\in C[0,1]$ 
we have the following for the typical continuous function:
{\it
 \begin{itemize}
\item there exists a denumerable dense set $S_f$ in $(min_f,Max_f)$ such that if
$ {\alpha} \not\in S_f \cup  \{min_f,Max_f\}$ then $L_{{\alpha},f}=P_{{\alpha},f}$ which is nowhere
dense and perfect,
\item $L_{{\alpha},f}$ is a single point if
$ {\alpha}\in \{min_f, Max_f\}$,
\item if $ {\alpha} \in S_f$ then $L_{{\alpha},f}$ is the union of a single point and
$P_{{\alpha},f} \ne \emptyset$, a nowhere dense and perfect set. This isolated point on the level set corresponds to a local extremum.
\end{itemize}
}

The set of functions satisfying the BG-property will be denoted by $BG$.

Actually the result in  \cite{BruGa} about the total level set structure of typical continuous functions is also very nice:
For $f\in C \left [0,1 \right]$ and $ {\gamma}\in  {\ensuremath {\mathbb R}}$ put $f_{{\gamma}}(x)=f(x)+ {\gamma} x$.
{\it For the typical continuous function there exists $ {\Gamma}_f$, a countable dense subset in
$ {\ensuremath {\mathbb R}}$ such that if $ {\gamma} \in  {\ensuremath {\mathbb R}} \setminus  {\Gamma}_f$, $f_{{\gamma}}$ has the BG-property and
if $ {\gamma} \in  {\Gamma}_f$ then $f_{{\gamma}}$ meets all
conditions of the BG-property, but there is one exceptional  level set which contains
two isolated points instead of one.}

Returning to Theorem \ref{T4} 
we remark that
all extrema of $f_{{\gamma}}$ with $ {\gamma}\ne 0$ belong to $E_f$ showing that
the exceptional
points are $c$-dense on the graph of $f$.

\subsection{Micro tangent sets of the Takagi function and of the Brownian motion.} 
\label{subsecmicrotak}

Recall from Section \ref{subsectak} the definition of the Takagi function. 
This is one of the well-known examples of nowhere differentiable functions,
however its H\"older spectrum is very simple, it is monofractal with Hölder exponent $1$ everywhere see \cite{Jaff97}
Section 6. This also means that although Takagi's function is nowhere differentiable, it is not too far from a differentiable function which has a Hölder exponent greater equal than $1$.

 \begin{theorem}[Theorem 10 of \cite{Bumicro}]\label{takagi}
For
almost every $x_{0}\in  {\ensuremath {\mathbb R}}$,
$(x_{0}, {{T}}(x_{0}))$ is a graph like micro tangent point of the Takagi
function, $ {{T}}(x)$. In fact, this function is ``micro self-similar"
in the sense that if we take $g= {{T}}|_{[-1,1]}$ then
$\graph(g)\in  {{T}}_{MT}(x_{0})$ for almost every $x_{0}\in  {\ensuremath {\mathbb R}}$.
\end{theorem}

In \cite{Bumicro} other functions were also studied.
Assume that $[W(t):t\geq 0]$ denotes the Brownian motion. Since the definition of a micro tangent set is local we can consider micro tangent points of functions in $C[0,+ {\infty}]$.
It is well-known that the Hausdorff dimension of the Brownian motion path is $3/2$  almost surely, this shows that the Brownian motion path is ``wilder'' then the graph of the typical continuous function which, as we mentioned earlier, is of Hausdorff dimension $1$. This is also reflected on the micro tangents sets:

 \begin{theorem}[Theorem 9 of \cite{Bumicro}] \label{bmo}
For almost every Brownian motion path $W(t),$ from $F\in W_{CMT}(t)$
$(t>0)$ it follows that $F {\subset} S_{0} {{\buildrel {\rm def} \over =}}  \{(0;y):|y|\leq 1  \}$.
Therefore, $CGLMT(W)= {\emptyset}$ and $UMT(W)= {\emptyset}$ with probability one.
\end{theorem}

Since $S_{0}$ is a vertical line segment passing through the origin, the above theorem tells that as we zoom in to any point $(t;W(t))$ on the graph of the Brownian motion it is leaving very quickly the squares $Q((t;W(t)),\delta)$.

We recall that with  probability one for the Brownian motion, $W$ the occupation measure
$\mu (A) =
{\mathcal {L}} (W^{-1}(A))$
is absolutely  continuous  with respect to the Lebesgue measure,
it satisfies the local time
(LT) condition, see for example \cite[Chapter 3]{MorPer}.

In general, occupation measures are often studied in the context of stochastic processes.
Bertoin \cite{Berom,Berhdlev} analyzed the occupation measures and Hausdorff dimensions of level sets for certain self-affine functions. Depending on parameter values, these functions either satisfy the (LT) condition or have a singular occupation measure. About occupation measures see also \cite{GemHor}.

\subsection{Weierstrass's nowhere 
differentiable function.}\label{subsecweierstrass} 
This subsection is mainly based on \cite{Buirra}
and \cite{buczoccupation2010}.
The classical Weierstrass type Cellerier nowhere differentiable function,
is defined by the Fourier series
\begin{equation}\label{**I5}
 {{\mathcal W}}(x) {{\buildrel {\rm def} \over =}} f(x)=\sum_{n=0}^{{\infty}}2^{-n}\sin(2\pi 2^{n}x).
\end{equation}

We also introduce a class
of ``perturbed" Weierstrass--Cellerier-type functions.
Denote by
$ {{\mathcal F}}_{{{\mathcal W}}}$ the set of those
twice continuously differentiable
functions $f_{-1}$ on $[0,1]$
for which for any trigonometric polynomial $P$
the function $f_{-1}+P$ is piecewise strictly monotone
or constant. If one takes a function which is
analytic
on an open set $G\supset [0,1]$ then it is in $ {{\mathcal F}}_{{{\mathcal W}}}$.
Suppose that $f_{-1}(x)$ belongs to $ {{\mathcal F}}_{{{\mathcal W}}}$.
For $x\in [0,1]$ put
\begin{equation}\label{*I5}
f(x)=f_{-1}(x)+\sum_{n=0}^{{\infty}}2^{-n}\sin(2\pi 2^{n}x).
\end{equation}

\begin{figure}[ht] \includegraphics[width=1\linewidth]{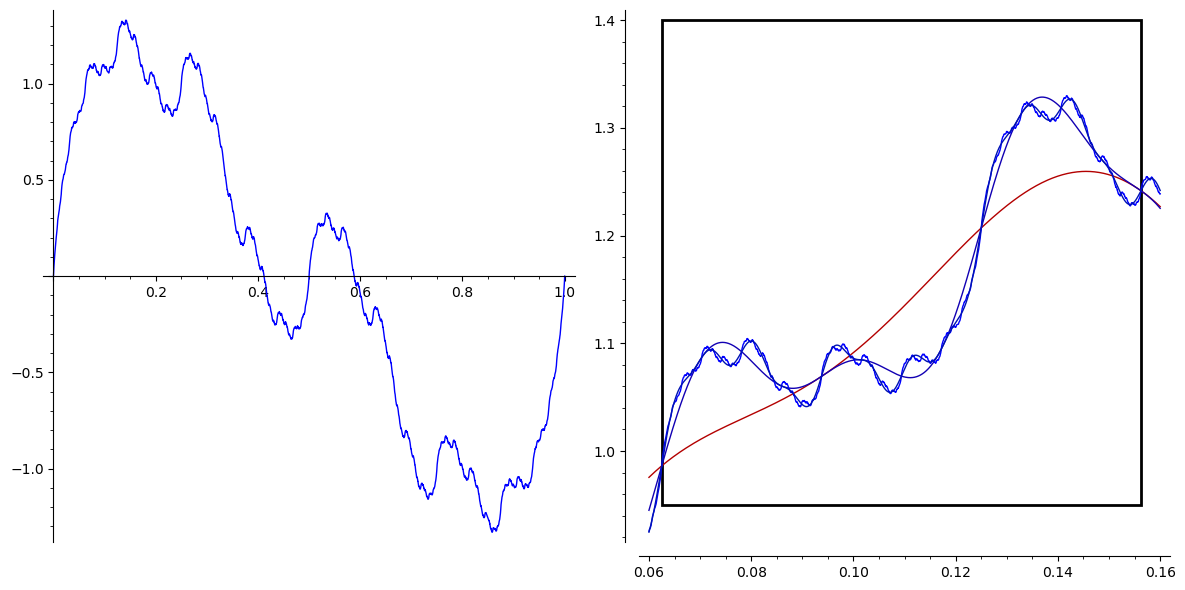} 
\caption{Left: The Weierstrass--Cellerier function, Right: Part of it with approximating ${\mathcal W}_{N}$  for $N=3,4,5,6$}\label{figrect}
\end{figure}

The partial sums in \eqref{*I5} are denoted by
$  \displaystyle  f_N(x)=f_{-1}(x)+\sum_{n=0}^{N}2^{-n}\sin(2\pi 2^{n}x).$

If $f_{-1}$ equals identically zero then we get back the function in \eqref{**I5}. In this special case
 we denote $f_{N}$
by $ {{\cal W}}_{N}$.

Set
$$ {X_ {1}}= \{x\in [0,1]:\liminf_{N\to  {\infty}} f_{N}'(x)=- {\infty},
\limsup_{N\to {\infty}}f_{N}'(x)=+ {\infty}   \}=
$$ $$ \{x\in [0,1]:\liminf_{N\to  {\infty}}  {{\cal W}}_{N}'(x)=- {\infty},
\limsup_{N\to {\infty}} {{\cal W}}_{N}'(x)=+ {\infty}   \}.$$

The results in \cite{Bumicro}  obtained for
$ {{\cal W}}$ imply that
$\liminf_{N\to  {\infty}} f_{N}'(x)=
\liminf_{N\to  {\infty}} (f_{N}-f_{-1})'(x)=- {\infty}$
and
$\limsup_{N\to  {\infty}} f_{N}'(x)=
\limsup_{N\to {\infty}}(f_{N}-f_{-1})'(x)=+ {\infty}$  for almost every $x$.
For these $x$s infinitely often $f'_{N}(x)$ and
$f'_{N+1}(x)$ are not of the same sign.
Since $$|f'_{N+1}(x)-f'_{N}(x)|=
|(2^{-N}\sin(2\pi 2^{N}x))'|=|2\pi \cos(2\pi 2^{N}x)|\leq
2\pi$$ for every $x\in  {X_ {1}}$ there exist infinitely many
$N$s such that $|f_{N}'(x)|\leq 2\pi$. There is also a reasonable bound for the second derivative $| {{\cal W}}''_{N}(x)|< 8\pi^{2}2^{N}.$ Using these properties one can infer a theorem analogous to the one for Takagi's function in 
Theorem \ref{thtakirrdec}.

\begin{theorem}[Theorem 13 of \cite{Buirra}]\label{irregthm}
The set $S_{irr}=  \{(x;f(x)): x\in  {X_ {1}}  \}$
is an irregular $1$-set, the set $S_{reg}=
 \{(x;f(x)):x\in[0,1] {\setminus}  {X_ {1}}   \}$ can be covered by the
union of the graphs of countably many strictly monotone
functions and by a set of zero $ {{\cal H} ^ {1}}$-measure.
\end{theorem}

If one wishes the set of zero measure can be added to
$S_{irr}$ and this way we can obtain a modified $S_{irr}^{*}$
and $S_{reg}^{*}$. Then the set $S_{reg}^{*}$ can be
covered by the
union of the graphs of countably many strictly monotone
functions.

Next one can see the way occupation measures are used to obtain a result analogous to the one about the finiteness of almost every level set of the Takagi function.

\begin{theorem}[Theorem 19 of \cite{Buirra}]\label{flevthm}
If $f_{-1}\in  {{\cal F}}_{{{\cal W}}}$
and the occupation measure of $f$ defined
in  \eqref{*I5} is singular then almost every
level set of $f(x)$ is finite.
\end{theorem}

The projection of the irregular $1$-set $S_{irr}$ onto the $x$-axis is again a set of
full measure. By the Besicovitch projection theorem in almost
every direction $S_{irr}$ projects into a set of Lebesgue
measure zero. This implies that
 $ {{\mathcal W}}(x,c) {{\buildrel {\rm def} \over =}} {{\mathcal W}}(x)+cx$ has purely singular
occupation measure for almost every $c$.
By Theorem \ref{flevthm}    for these $c$
the level sets, $L_{y}$ of $ {{\mathcal W}}(x,c)$ are finite for almost every
$y$. Of course, out of the almost every values the case $c=0$, the case of the Weierstrass--Cellerier function is the most interesting.

There are many results where it can be proven
that a certain property holds for almost every
parameter value, but it is much more difficult to see that a certain
parameter value satisfies this property. My favorite example is the following.
It is well-known that  $\{\theta^n\}$ (the fractional part of $\theta^n$) is uniformly distributed $\mod 1$
for almost every $\theta > 1$ see for example \cite[Theorem 4.5.2]{Bertin}. But there is the famous unsolved problem about density and uniform distribution of $\{(3/2)^n\}$  $\mod 1$.

Fortunately, in the case of the functions ${{\mathcal W}}(x,c)$  it has turned out that there are no exceptional $c$s.

\begin{theorem}[Theorem 1 of \cite{buczoccupation2010}]\label{thmmainwoc}
For all $c\in { \ensuremath { \mathbb R } }$ the occupation measure of the function $ { { \cal W } }(x,c)$ is purely singular.
\end{theorem}

 By Theorem \ref{flevthm}  this also implies that for any $c$ almost every level set  is finite for   $ { { \cal W } }(x,c)$, including the case $c=0$, the classical Weierstrass--Cellerier function.

\begin{figure}[ht!]
\centering
\includegraphics[height=0.49\linewidth]{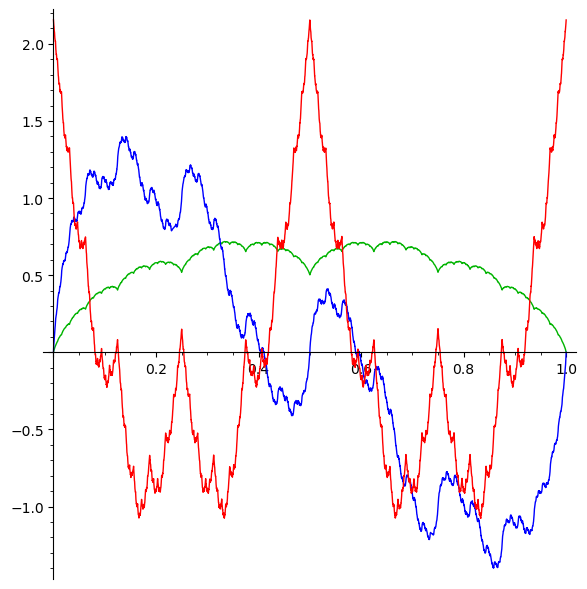}
\includegraphics[width=0.49\linewidth]{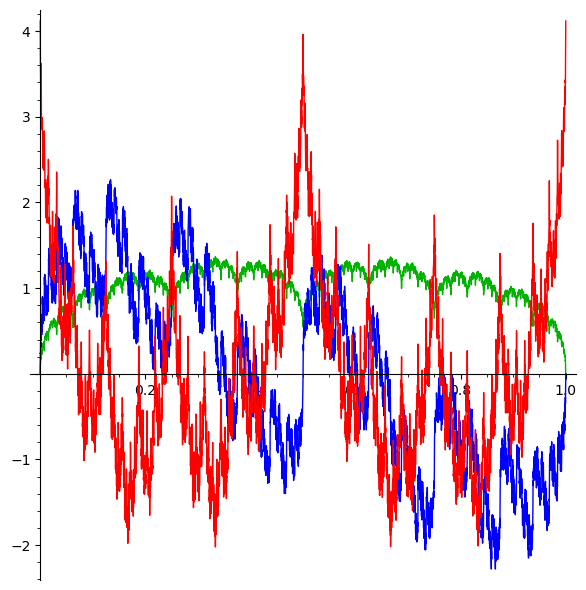} 
\captionsetup{width=0.9\textwidth}
\caption{Three Weierstrass functions, on the left with $\aaa=0.9$ on the right with $\aaa=0.4$}
\label{fig-alphalarge0904}
\end{figure}

\subsection{Level sets and occupation measures of prevalent Weierstrass functions.} 
\label{subwlev}
This section is about the recent paper \cite{buczolich2025levelsets}.
In this paper a much larger class of Weierstrass-type functions was studied and instead of the topological notion of typicality the measure theory based prevalence was considered.
We studied level sets and occupation measures of prevalent Weierstrass functions.

Prevalence was introduced in \cite{HuSaYoprev}, see also \cite{OtYoprev}. For the sake of this survey paper we introduce only the concept of $d$-prevalence. Every $d$-prevalent set is also prevalent in the general case and often verifying the special case of $d$-prevalence is the way of showing that a property is prevalent. Suppose that $X$ is a Banach space a subset $A\sse X$ is $d$-prevalent if one can find a $d$-dimensional subspace $S \subset X$, a suitable \emph{probe space}, such that for any $x \in X$, we have $x + s \in B$ for Lebesgue-almost every $s \in S$, where the Lebesgue measure
$ {\mathcal {L}}^{d}$ is defined on $S$ and $B$ is a Borel set contained in $A$.

For an integer $b \ge 2$, $0<\alpha\leq 1$, and a non-constant Lipschitz function $g$ periodic by $1$ the \emph{Weierstrass function} $W_g^{\alpha,b} \colon \mathbb{T} \to  {\mathbb {R}}$ is defined by
\begin{equation} \label{*Wdef}
  W_g^{\alpha,b}(x) = \sum_{k=0}^{\infty} b^{-\alpha k} g(b^kx).
\end{equation}
For fixed $b \ge 2$ and $0 < \alpha \leq 1$, the family of these Weierstrass functions is denoted by $ {\mathcal {W}}^{\alpha,b}$.

Recall that  $f\colon [0,1]\to  {\mathbb {R}}$ is \emph{$\alpha$-H\"older} if  for all $x,y\in [0,1]$ there exists a constant $C>0$ such that
\begin{equation} \label{eqholder} 
  |f(x)-f(y)|\leq C|x-y|^\alpha.
\end{equation}

 The Weierstrass functions in $ {\mathcal {W}}^{\alpha,b}$ are known to be $\alpha$-H\"older (see, e.g., \cite[Proposition~2.3]{10.1007/978-3-319-18660-3_5}). 
The functions in $\mathcal{W}^{\alpha,b}$ are in one-to-one correspondence with the Banach space $\mathrm{Lip}(\mathbb{T})$ of Lipschitz functions defined on the unit circle $\mathbb{T}$. In this Banach space we use the norm
$\ds   \|g\|_{\mathrm{Lip}} = \sup_{x \in \mathbb{T}}|g(x)| + \sup_{(x,y) \in \mathbb{T}\times \mathbb{T},\ x\not=y}  \frac{|g(x)-g(y)|}{|x-y|}.
$
Thus, we identify $ {\mathcal {W}}^{\alpha,b}$ with $\mathrm{Lip}(\mathbb{T})$, and prevalence in $ {\mathcal {W}}^{\alpha,b}$ refers to prevalence under this identification. The family $ {\mathcal {W}}^{\alpha,b}$ includes well-studied examples such as the classical Weierstass function (see for example \cite{Harwei}), where $g(x) = \cos(2\pi x)$,
the Weierstrass--Cellerier function with $g(x) = \sin(2\pi x)$ and
$\aaa=1$ and $b=2$,
 and the Takagi function, where $g(x) = \dist(x, {\mathbb {Z}})$ also with $\aaa=1$ and $b=2$. On Figure \ref{fig-alphalarge0904} one can see the versions of the previous functions when $\aaa=0.9$ and $\aaa=0.4$ are used in Definition \ref{*Wdef}. The figure uses different scales on the $x$ and $y$ axes, the generalized classical Weierstrass function is in red and it takes non-zero positive value at $0$. The generalized Takagi function is in green and is taking non-negative values. Finally, the third function in blue is the generalized Weierstrass--Cellerier function. The figure illustrates that the smaller $\aaa$ the wilder are the functions.

We say that the one-to-one map $\Phi \colon [0,1] \to  {\mathbb {R}}^d$ is an \emph{$\alpha$-bi-H\"older map} if $\Phi $ and its inverse are both $\aaa$-Hölder. 
In \cite{buczolich2025levelsets} we had to find  an $\alpha$-bi-H\"older embedding $\Phi\colon \mathbb{T} \to  {\mathbb {R}}^d$, whose coordinate functions are in $ {\mathcal {W}}^{\alpha,b}$.

Let $\mathcal{G}=\{g_0,\ldots,g_{d-1}\}$ be a finite collection of Lipschitz functions on $\mathbb{T}$. The \emph{Weierstrass embedding} $\Phi_{\mathcal{G}}^{\alpha,b}\colon \mathbb{T}\to {\mathbb {R}}^d$ associated to the collection $\mathcal{G}$, integer $b \ge 2$, and $0<\alpha<1$ is defined by
\begin{equation} \label{*Web}
  \Phi_{\mathcal{G}}^{\alpha,b}(x)=(W_{g_0}^{\alpha,b}(x),W_{g_1}^{\alpha,b}(x),\ldots,W_{g_{d-1}}^{\alpha,b}(x)).
\end{equation}

One of the key tools of \cite{buczolich2025levelsets} was the following theorem about the existence of  {Weierstrass embeddings}.

\begin{theorem}[Theorem 2.3 of \cite{buczolich2025levelsets}] \label{prop:bi-holder}
  For every integer $b \ge 2$ and $0<\alpha<1$ there exist $d \in  {\mathbb {N}}$ and a finite collection $\mathcal{G}=\{g_0,\ldots,g_{d-1}\}$ of Lipschitz functions on $\mathbb{T}$ such that the Weierstrass embedding $\Phi_{\mathcal{G}}^{\alpha,b}\colon \mathbb{T}\to {\mathbb {R}}^d$ is $\alpha$-bi-H\"older.
\end{theorem}

It is an interesting question to determine, or at least give an estimate about the minimal dimension $d$ in the above theorem. The next proposition shows that the smaller the Hölder exponent $\aaa$, the more functions we need for the Weierstrass embedding, that is we need larger $d$. However we do not know how close is the given estimate to the best possible value of $d$ corresponding to the given $\aaa$.

\begin{proposition}[Proposition 2.4  of \cite{buczolich2025levelsets}]\label{*propaaabound}
  Let $b \geq 2$ and $d \geq 1$ be integers, and let $0 < \alpha < 1$. If $d <  \frac{1}{\alpha}$ and $\mathcal{G} = \{g_0, \ldots, g_{d-1}\}$ is a finite collection of Lipschitz functions on $\mathbb{T}$, then the Weierstrass embedding $\Phi_{\mathcal{G}}^{\alpha,b} \colon \mathbb{T} \to \mathbb{R}^d$ is not $\alpha$-bi-H\"older.
\end{proposition}

The fact that for smaller $\aaa$ we need more functions is illustrated on Figure \ref{fig-alphalargeb}. The functions from Figure \ref{fig-alphalarge0904}  were taken and in the plane those points $x,y$ were colored where for the given function \eqref{eqholder} does not hold with a constant $C=0.6$. For the Weierstrass embedding one needs that there are no points which are colored by all colors simultaneously.
Of course many different positive $C$s can be used.
Zero loci of the functions, that is where the left hand side of \eqref{eqholder} is zero are bad for any $C$. These points are colored a little darker on Figure \ref{fig-alphalargeb}. The two sides of the figure illustrate that the smaller $\alpha$ we get more complicated regions and it seems to be more likely that we get points colored by all three colors. Of course, these images are not proving anything and if somebody zooms in near the diagonal starting at $(0,0)$ they can get very complicated.

\begin{figure}[ht!]
\centering
\includegraphics[height=0.49\linewidth]{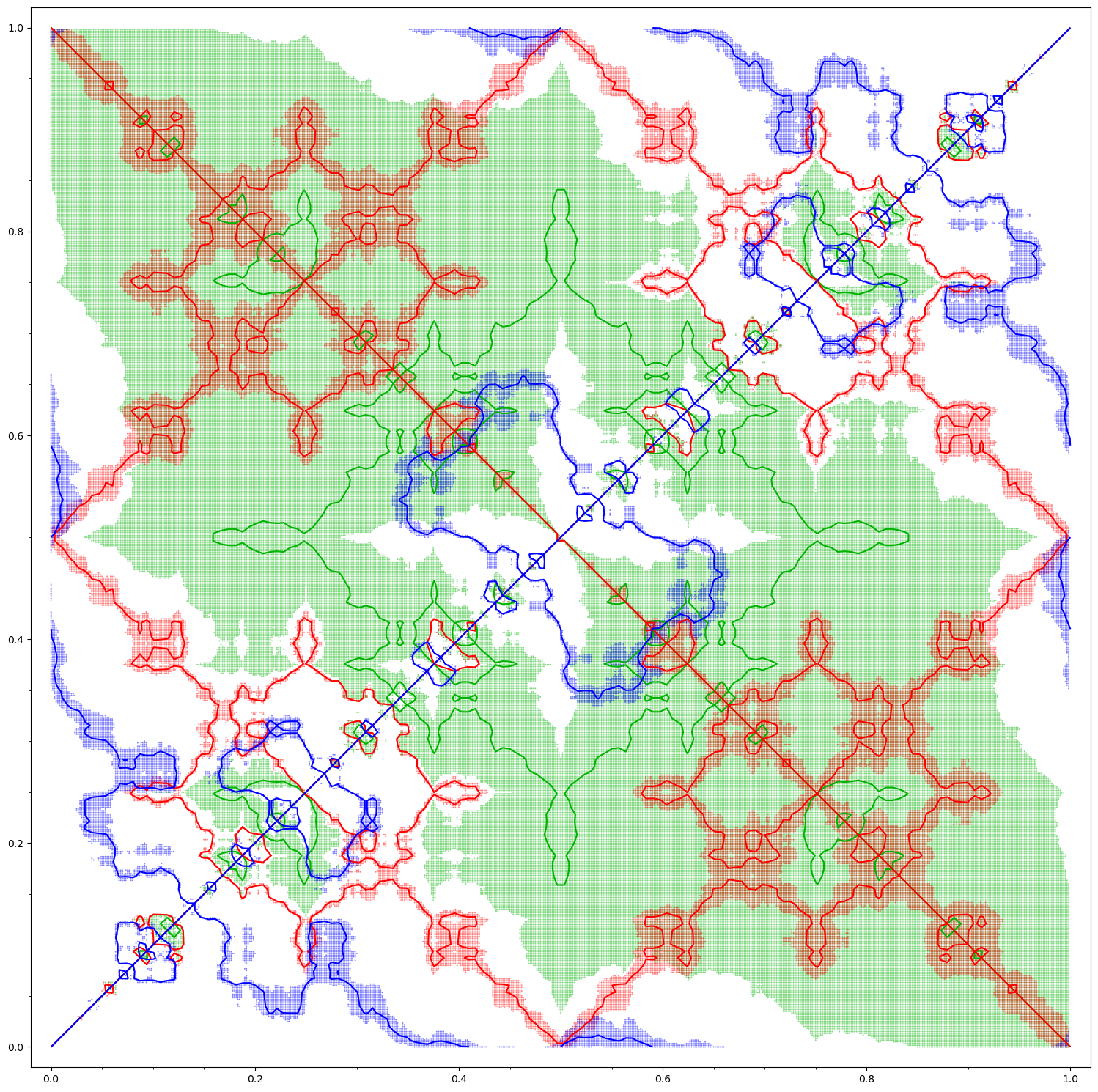}
\includegraphics[width=0.49\linewidth]{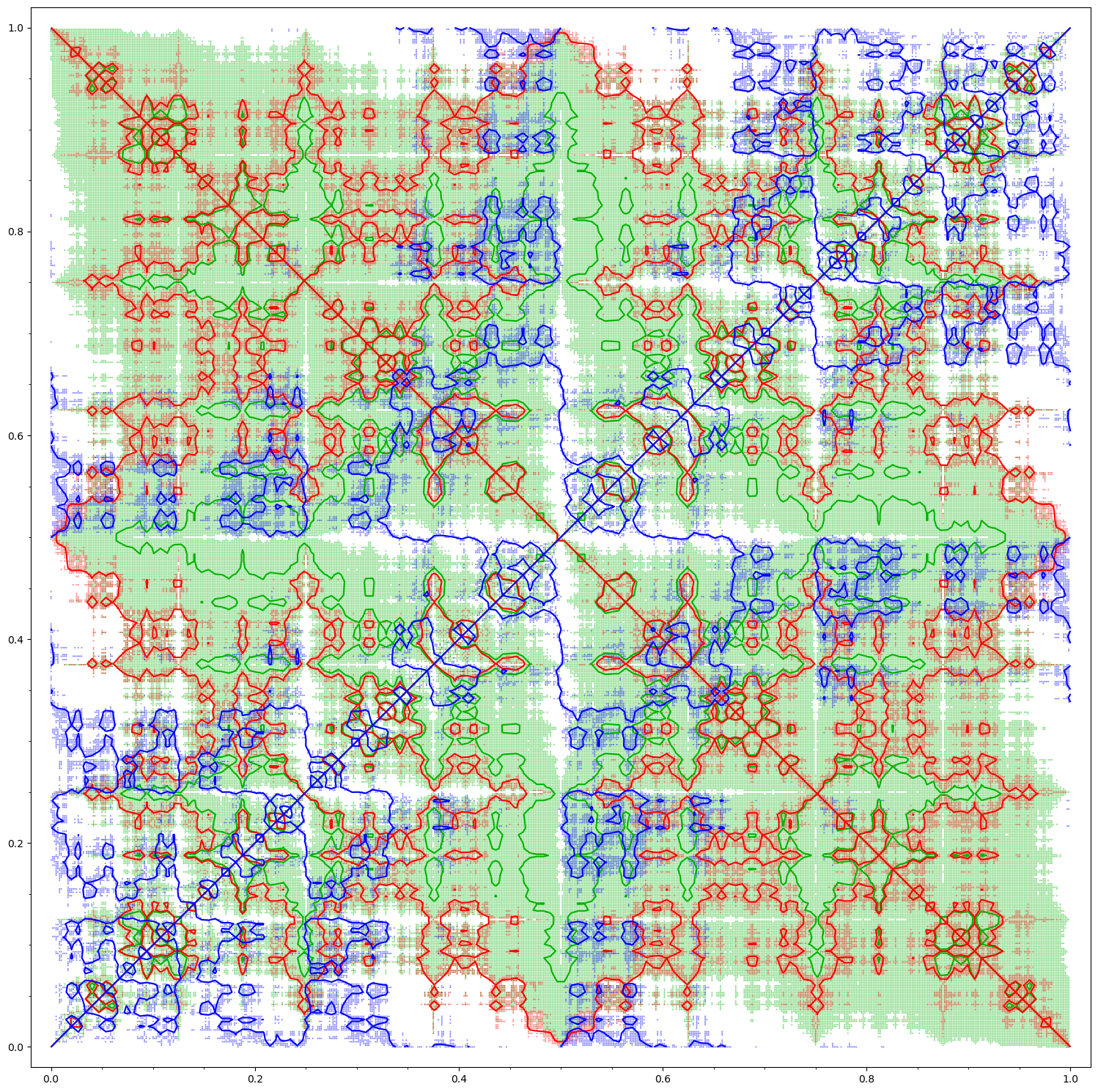} 
\captionsetup{width=0.9\textwidth}
\caption{For the functions of Figure \ref{fig-alphalarge0904}  regions where \eqref{eqholder} does not hold with $C=0.6$.  Left: $\aaa=0.9$ Right: $\aaa=0.4$}
\label{fig-alphalargeb}
\end{figure}

To prove prevalence results Weierstrass embedding was used the following way. Suppose that $b \ge 2$ is an integer and $0<\alpha<1$. 
Using Theorem \ref{thm:main1} take
an $\alpha$-bi-H\"older Weierstrass embedding
  $\Phi = \Phi_{\mathcal{G}}^{\alpha,b}$  associated with the system $\mathcal{G} = \{g_0,\ldots,g_{d-1}\}$. For any $W \in  {\mathcal {W}}^{\alpha,b}$ and $ {\mathbf {t}} \in  {\mathbb {R}}^d$, the function $W_{{\mathbf {t}}} \colon \mathbb{T} \to  {\mathbb {R}}$, defined by
\begin{equation*}
  W_{{\mathbf {t}}}(x)=W(x)+ \langle {\mathbf {t}},\Phi(x) \rangle,
\end{equation*}
belongs to $ {\mathcal {W}}^{\alpha,b}$, (here $\langle .,.\rangle $ denotes the scalar product). 
We denote by $ {\lll}_ {\mathbf {t}}$ the occupation measure of $W_{{\mathbf {t}}}$, that is $ \lambda_ {\mathbf {t}}(A)={\mathcal {L}}\{x : W_{{\mathbf {t}}}(x)\in A\}$.

We had to study the absolute continuity of the occupation measure of $W_{{\mathbf {t}}}$
and showed in the next theorem that the occupation measure $ \lambda_ {\mathbf {t}}$ is absolutely continuous for almost every $ {\mathbf {t}}$.

\begin{theorem}[Theorem 3.1 of \cite{buczolich2025levelsets}] \label{*prabscb}
  Let $b \ge 2$ be an integer, $0<\alpha<1$, and $W \in  {\mathcal {W}}^{\alpha,b}$. Then for $ {\mathcal {L}}^d$-almost every $ {\mathbf {t}} \in  {\mathbb {R}}^d$ the occupation measure $ \lambda_ {\mathbf {t}}$ associated with $W_ {\mathbf {t}}$ is absolutely continuous with density in $L^2( {\mathbb {R}})$.
\end{theorem}

This implies  that the $d$-prevalent $\alpha$-Weierstrass functions in $\mathcal{W}^{\alpha,b}$  satisfy the local time (LT) condition. 

It is worth comparing this theorem to Remark \ref{defoccmeas} since  we saw that the occupation measure is singular for the Takagi function and  for   $W(x) + cx$ for all $c$. This does not contradict the above theorem, since  the Weierstrass-Cellerier function and the Takagi function belong to the class $\mathcal{W}^{1,2}$,
while Theorem \ref{*prabscb} is about prevalent functions in the spaces $\mathcal{W}^{\alpha,b}$, with $0 < \alpha < 1$.

For the range $0 < \alpha <  \frac{1}{2}$ scrutinizing the proof  of \cite[Proposition 3.2]{AnttilaBaranyKaenmaki2025} one can see that for the prevalent $\aaa$-Hölder function a stronger result holds: the occupation measure is not only absolutely continuous but also has a bounded and continuous density. By combining ideas from the proof  of \cite[Proposition 3.2]{AnttilaBaranyKaenmaki2025} with our methods we proved the following theorem.

\begin{theorem}[Theorem 3.2 of \cite{buczolich2025levelsets}] \label{*prabscbb}
  Let $b \ge 2$ be an integer, $0<\alpha< \frac12$, and $W \in  {\mathcal {W}}^{\alpha,b}$. Then for $ {\mathcal {L}}^d$-almost every $ {\mathbf {t}} \in  {\mathbb {R}}^d$ the occupation measure $ \lambda_ {\mathbf {t}}$ associated with $W_ {\mathbf {t}}$ is absolutely continuous with bounded and continuous density.
\end{theorem}

Our main goal in \cite{buczolich2025levelsets} was to prove a theorem analogous to 
the main result of \cite{AnttilaBaranyKaenmaki2025} about prevalent $\aaa$-Hölder functions which was the following:

\begin{theorem}[Theorem~1.6 of \cite{AnttilaBaranyKaenmaki2025}]  \label{thm:main_holder}
  A prevalent $\alpha$-H\"older function $f$ on the unit interval satisfies
  \begin{enumerate}[label=(\roman*)]
    \item \label{it:main1_holder} $\udimm(f^{-1}(\{y\})) \le 1-\alpha$ for all $y \in  {\mathbb {R}}$ provided that $0<\alpha< \tfrac12$,
    \item \label{it:main2_holder} $ {\mathcal {L}}^1(\{y \in f([0, 1]): \dimh(f^{-1}(\{y\})) = 1-\alpha\}) > 0$ provided that $0<\alpha<1$.
  \end{enumerate}
\end{theorem}

We proved  an exact analogue of Theorem \ref{thm:main_holder} for prevalent $\alpha$-Weierstrass functions in the following two theorems:

\begin{theorem}[Theorem 1.2  of \cite{buczolich2025levelsets}] \label{thm:main1}
  For any integer $b \ge 2$, a prevalent function $g \in \mathrm{Lip}(\mathbb{T})$ satisfies
  \begin{equation*}
    \udimm((W_g^{\alpha,b})^{-1}(\{y\})) \le 1-\alpha
  \end{equation*}
  for all $y \in  {\mathbb {R}}$ provided that $0<\alpha< \tfrac12.$
\end{theorem}

The existence of the  $\alpha$-bi-H\"older Weierstrass embeddings and methods from \cite{AnttilaBaranyKaenmaki2025} were sufficient to prove \ref{thm:main1}.

To prove Theorem \ref{thm:main_holder}  \ref{it:main2_holder}
in
\cite{AnttilaBaranyKaenmaki2025} an ``almost everywhere method''  was used.
First  slices in almost every direction were considered. To obtain the result for the specific horizontal slice corresponding to the level sets, a linear function was added to the H\"older function. In the Banach space $ {\mathcal {W}}^{\alpha,b}$, which consists of $1$-periodic functions, this trick is not working. To prove the prevalent Weierstrass version of Theorem \ref{thm:main_holder}  \ref{it:main2_holder}
different methods were used to obtain:

\begin{theorem}[Theorem 1.3  of \cite{buczolich2025levelsets}] \label{thm:main2}
  For any integer $b \ge 2$, a prevalent function $g \in \mathrm{Lip}(\mathbb{T})$ satisfies
  \begin{equation*}
     {\mathcal {L}}^1(\{y \in W_g^{\alpha,b}(\mathbb{T}) : \dimh((W_g^{\alpha,b})^{-1}(\{y\})) = 1-\alpha\}) > 0
  \end{equation*}
	provided that $0<\alpha<1$.
\end{theorem}

In \cite{buczolich2025levelsets} there are several open questions  here we just mention one of them.
Instead of focusing solely on the occupation measure, one can consider slices and projections in various directions. For $\theta \in [0, 2\pi)$, let ${\mathrm {pr}}_\theta \colon \mathbb{R}^2 \to \mathbb{R}$ denote the orthogonal projection defined by
  \begin{equation*}
    {\mathrm {pr}}_\theta(x, y) = x \cos(\theta) + y \sin(\theta)
  \end{equation*}
  for all $(x, y) \in \mathbb{R}^2$. Note that $\proj_2 = \Pi_{y} = {\mathrm {pr}}_{\frac{\pi}{2}}$.
  Put
\begin{equation}  \label{eq:lift-measure}
  \mu_{{\mathbf {t}}}=(\id,W_{{\mathbf {t}}})_* {\mathcal {L}}^1.
\end{equation}
 One can consider the projections of $\mu_ {\mathbf {t}}$ for $\theta \in [0, 2\pi)$, defined as
  \begin{equation*}
     \lambda_ {\mathbf {t}}^\theta = ({\mathrm {pr}}_\theta)_* \mu_ {\mathbf {t}},
  \end{equation*}
  (where $*$ denotes the push-forward).
  These measures $ \lambda_ {\mathbf {t}}^\theta$ can be called as \emph{$\theta$-oblique occupation measures}. Note that the occupation measure is $ \lambda_ {\mathbf {t}} =  \lambda_ {\mathbf {t}}^{\frac{\pi}{2}}$. 

Using methods from   \cite{buczolich2025levelsets} and \cite{AnttilaBaranyKaenmaki2025}, one can show that for prevalent functions in $\mathcal{W}^{\alpha,b}$, the $\theta$-oblique occupation measure is absolutely continuous with respect to the Lebesgue measure for Lebesgue almost every $\theta$. 
Here again we are in a situation when one can conjecture that in the previous statement almost every $\theta$ can be replaced by all $\theta$, that is, whether the $\theta$-oblique occupation measure is absolutely continuous for every $\theta$. This means that we can do an improvement similar to the one which was leading to Theorem \ref{thmmainwoc}. See also the paragraph following Theorem \ref{flevthm}.

\bibliographystyle{alpha}
\bibliography{wlsum}

\end{document}